\documentclass{article}

\usepackage{amsmath, amsthm, amssymb}
\usepackage{color,xcolor}
\newtheorem{lemma}{Lemma}

\newtheorem{proposition}[lemma]{Proposition}
\newtheorem{remark}[lemma]{Remark}

\newtheorem{theorem}[lemma]{Theorem}

\newcommand{\e}{\varepsilon} 

\title{Higher-order non-local gradient theory\\ of phase-transitions}
\author{Margherita Solci\footnote{email: {\tt margherita@uniss.it}}\\ 
{\small DADU, Universit\`a di Sassari,} {\small piazza Duomo 6, 07041 Alghero (Italy)}}
\date{}
\begin{document} 

\maketitle
\begin{abstract} 
We study the asymptotic behaviour of double-well energies perturbed by a higher-order fractional term, which, in the one-dimensional case, take the form
$$
\frac1\e\int_I W(u(x))dx+\e^{2(k+s)-1}\frac{s(1-s)}{2^{1-s}}\int_{I\times I} \frac{|u^{(k)}(x)-u^{(k)}(y)|^2}{|x-y|^{1+2s}} dx\,dy
$$
defined on the higher-order fractional Sobolev space $H^{k+s}(I)$,
where $W$ is a double-well potential, $k\in \mathbb N$ and $s\in(0,1)$ with $k+s>\frac12$. 
We show that these functionals $\Gamma$-converge to a sharp-interface functional with domain $BV(I;\{-1,1\})$ of the form $m_{k+s}\#(S(u))$, with $m_{k+s}$ given by the optimal-profile problem
\begin{eqnarray*}
&&m_{k+s} =\inf\Big\{\int_{\mathbb R} W(v)dx+\frac{s(1-s)}{2^{1-s}}\int_{\mathbb R^2}\frac{|v^{(k)}(x)-v^{(k)}(y)|^2}{|x-y|^{1+2s}} dx\,dy :\\
&& \hskip5cm v\in H^{k+s}_{\rm loc}(\mathbb R), \lim_{x\to\pm\infty}v(x)=\pm1\Big\}.
\end{eqnarray*}
The normalization coefficient $\frac{s(1-s)}{2^{1-s}}$ is such that $m_{k+s}$ interpolates continuously the corresponding $m_k$ defined on standard higher-order Sobolev space $H^k(I)$, obtained by Modica--Mortola in the case $k=1$ \cite{MM}, Fonseca--Mantegazza in the case $k=2$ \cite{FM} and Brusca, Donati and Solci  for $k\ge 3$ \cite{BDS}. The results also extends previous works by Alberti, Bouchitt\'e and Seppecher \cite{ABS}, Savin--Valdinoci \cite{SV} and Palatucci--Vincini \cite{PaVi} in the case $k=0$ and $s\in(\frac12,1)$.
\smallskip



\end{abstract}

\section{Introduction}
Variational models of phase transitions consider integrals depending on double-well potentials $W$ with two (or more) minimizers. These simple models, as such admit highly discontinuous unphysical solutions; a selection of minimizers can be ensured by a singular-perturbation approach in which a higher-order term is added, the form of which is sometimes dictated by modeling considerations. In the Cahn--Hilliard theory of phase separations \cite{CH} (but earlier models leading to similar formulations date back to Landau \cite{Landau}, and even to van der Waals \cite{vdW}) the corresponding energy is of the form
$$
\int_\Omega W(u)dx+\e^2\int_\Omega|\nabla u|^2dx,
$$
with a first-order perturbation depending on the small parameter $\e$. Analytically, the description of the solutions for $\e$ small translates in an asymptotic analysis, which leads to a sharp-interface model in which the interaction of the two integral terms gives rise to an interfacial energy separating two phases characterized by the minimization of $W$, and minimum problems for $\e>0$ are approximated by corresponding minimum problems for interfacial energies. This {\em minimal-interface criterion} was conjectured e.g.~by Gurtin  \cite{gurtin}. Its rigorous statement was made possible by the development of the theory of sets of finite perimeter and of functions of bounded variation, which allowed to rigorously define the limit interfacial functionals. The proof of the minimal-interface criterion was then achieved by Modica \cite{Modica}, following earlier results in a seminal paper by Modica and Mortola \cite{MM}, showing the $\Gamma$-convergence of the scaled energies
$$
\frac1\e\int_\Omega W(u)dx+\e\int_\Omega|\nabla u|^2dx
$$
to an interfacial energy, which is simply given by
$$
m_W {\rm Per}(E,\Omega),
$$
where $E$ is a set of finite perimeter, and ${\rm Per}(E,\Omega)$ denotes the perimeter of $E$ in $\Omega$. In the case where the minimizers of $W$ are $-1$ and $1$ (which can be assumed up to a change of variables), minimal functions $u_\e$ can be described as approximating  $u=-1+2\chi_E$, where $E$ is a minimal set.
In this process, an interesting observation is that functions $u_\e$ approximating $\chi_E$ tend to follow an optimal profile in the direction orthogonal to the interface.
This remark is highlighted by the fact that the constant $m_W$ can be expressed as a one-dimensional {\em optimal-profile problem}
$$
m_W=\inf\Big\{\int_{-\infty}^{+\infty}(W(v)+ |v'|^2)dt: v\in H^1_{\rm loc}(\mathbb R),  \lim_{t\to\pm\infty} v(t)=\pm1\Big\}.
$$
The Modica result has been generalized and revisited in many ways and in different contexts, leading also to more complex non one-dimensional descriptions (see e.g.~\cite{MR1051228,MR1286918,MR1036589,MR1097327}). However, even when dealing with higher-dimensional problems, a reduction to the one-dimensional case is often useful to ensure compactness by a slicing argument (see e.g.~\cite{alberti,bln}). We note moreover that the analysis of perturbed double-well energies is the building block for many ``phase-field'' approximation procedures of free-discontinuity problems starting from the seminal work by Ambrosio and Tortorelli \cite{AT} (see \cite{bln} for a presentation of different approximations).

\smallskip
The higher-order perturbation term of the double-well energy can be taken of different forms involving {\em higher-order derivatives}, justified either by the necessity of higher regularity (useful for the approximation procedures mentioned above), or by modeling assumptions (e.g.~from atomistic theories with long-range interactions as in \cite{BCST}). In this case the energies take the form
$$
\int_\Omega W(u)dx+\e^{2k}\int_\Omega\|\nabla^{(k)} u\|^2dx,
$$
the second term denoting a norm of the $k$-th order tensor of derivatives of order $k$. Such functionals have been considered by Fonseca and Mantegazza \cite{FM} for $k=2$ and by Brusca, Donati, and Solci \cite{BDS} for arbitrary $k$. They showed that, taking $\|\cdot\|$ the operator norm, the scaled functionals 
$$
\frac1\e\int_\Omega W(u)dx+\e^{2k-1}\int_\Omega\|\nabla^{(k)} u\|^2dx,
$$
still $\Gamma$-converge to a sharp-interface energy $m_k{\rm Per}(E,\Omega)$, where now the (one-dimensional) optimal-profile problem giving $m_k$ is
\begin{equation}\label{mint-k}
m_k=\inf\Big\{\int_{-\infty}^{+\infty}(W(v)+ |v^{(k)}|^2)dt: v\in H^k_{\rm loc}(\mathbb R),  \lim_{t\to\pm\infty} v(t)=\pm1\Big\}.
\end{equation}
We note that in this case even the proof of the inequality $m_k>0$ is not trivial if $k\neq 1$. 

In a different direction, it is also natural in some problems to consider perturbations of the form 
$$
\int_\Omega W(u)dx+\e^{2s}\int_{\Omega\times\Omega} \frac{|u(x)-u(y)|^2}{|x-y|^{d+2s}} dx\,dy,
$$
where $s\in(0,1)$ and $\Omega\subset \mathbb R^d$. The additional term is a Gagliardo {\em fractional $s$-seminorm}, which can appear as a boundary term from a higher-dimensional problem (see e.g.~\cite{ABS-2}). Under different scaling, these problems have been studied by Savin and Valdinoci \cite{SV} (see also Palatucci and Vincini \cite{PaVi}, earlier results by Alberti, Bouchitt\'e and Seppecher \cite{ABS} and Alberti and Bellettini \cite{AB}, and a recent review paper by Dipierro and Valdinoci \cite{DipVal}). In particular if $2s>1$ they note that the scaled energies
$$
\frac1\e\int_\Omega W(u)dx+\e^{2s-1}\int_{\Omega\times\Omega} \frac{|u(x)-u(y)|^2}{|x-y|^{d+2s}} dx\,dy,
$$
still define a singularly-perturbed functional, and show their convergence to a sharp-interface functional. 

\smallskip
In this paper we give a complete description which comprises both the higher-order and the fractional cases. We restrict our analysis to the one-dimensional case, which contains the relevant new technical ideas for the proofs. The treatment of the $d$-dimensional case is postponed to future work, since it requires additional techniques such as slicing (see \cite{bln,alberti}), which allows one to obtain coerciveness from the one-dimensional analysis, and blow-up (see \cite{FMu}), which leads to a description of the surface energy density as a limit of minimum problems on cubes of diverging size (see in particular the work by Dal Maso, Fonseca and Leoni \cite{MR3748585} where the representation formula is obtained for a similar problem). The higher-dimensional issues are overall standard but rather complex and technically challenging, and need the use of some adaptations to the fractional case (see for example the fractional slicing result \cite[Theorem 6.47]{leofrac}). Note that in the higher-dimensional case in the limit we will have an anisotropic perimeter functional 
depending on the norm that we choose in the terms of the form $\|D^{(k)}u(x)-D^{(k)}u(y)\|$.

The main result of the paper is the following.
We consider energies defined in the higher-order fractional Sobolev space $H^{k+s}$ (see e.g.~the recent monograph \cite{leofrac} and the `guide' \cite{DPV}). We suppose that $W$ is a non-negative continuous function with minima in $-1$ and $1$, and larger than a constant at infinity.
For fixed $k\in \mathbb N$ and $s\in(0,1)$ with $k+s>\frac12$ (so that the non-local part is still a singular perturbation after scaling by $1/\e$) and a bounded interval $I$, we consider $F_\e$ defined by
$$
F_\e(u)=\frac1\e\int_I W(u(x))dx+\e^{2(k+s)-1}\int_{I\times I} \frac{|u^{(k)}(x)-u^{(k)}(y)|^2}{|x-y|^{1+2s}} dx\,dy,
$$
and show that the $\Gamma$-limit of $F_\e$ as $\e\to 0$ with respect to the convergence in measure and in $L^1(0,1)$ is the sharp-interface functional given by 
$F(u)= m_{k+s}\# S(u)$, where $S(u)$ is the set of discontinuity points of $u$ defined on $BV(I;\{-1,1\})$, with
\begin{eqnarray}\label{emme}\nonumber
&&m_{k+s} =\inf\Big\{\int_{\mathbb R} W(v)dx+\int_{\mathbb R^2}\frac{|v^{(k)}(x)-v^{(k)}(y)|^2}{|x-y|^{1+2s}} dx\,dy :\\
&& \hskip4.5cm v\in H^{k+s}_{\rm loc}(\mathbb R), \lim_{x\to\pm\infty}v(x)=\pm1\Big\}
\end{eqnarray}
defining a strictly positive constant.

\smallskip
The proofs of the result mainly rely on the interpolation techniques first developed in \cite{solci} to treat the more complex case of singular perturbations producing free-discontinuity problems, and subsequently adapted to phase-transition problems in \cite{BDS}. Since the technical results in those papers are tailored for local integral energies, they must be modified for the present situation with non-trivial changes. The techniques mainly focus on estimating the size of intervals where functions are close to $\pm 1$, and at the same time have large derivatives of all orders between $1$ and $k$, in terms of the energy on the corresponding interval. Since this energy contains the $s$-seminorm of the $k$-th derivative and an integral that can be estimated by the $L^2$-distance from a well, suitable interpolation inequalities must be proved in terms of these quantities. These estimates allow to provide bounds in terms of the whole $H^{k+s}$-norm, showing that either a function oscillates, and hence contains points where derivatives of each order vanish, or has small derivative of all orders at some point close to the wells. In both cases, we can use a Poincar\'e-inequality argument to bound the total $H^{k+s}$-norm (and hence ensure that compactness properties hold). This allows one to bound the number of transitions between the wells along sequences $u_\e$ for which  $F_\e(u_\e)$ is equibounded, and hence to prove the equicoerciveness of the energies. Note that, differently from \cite{BDS}, where the existence of points with small derivatives and close to the wells was used to extend functions outside an interval to prove a lower bound of the energies, in the non-local case such an argument is not sufficient, and we use instead a cut-off argument together with an extensive use of interpolation inequalities. 

We have taken care to use minimal hypotheses on $W$; namely only that it has a quadratic behaviour at $\pm1$ and is bounded at infinity. The same arguments can be applied if $W(z)$ behaves as $||z|-1|^q$ close to $\pm1$ with $q>1$, by modifying the interpolation inequalities. We have limited our analysis to $q=2$ for ease of notation. Moreover, the compactness result is trivially improved to convergence in some $L^p$ if $W$ has a polynomial growth at infinity.

\medskip
If we compare the fractional case as stated above with the integer case, we note a divergence issue at integer points,
highlighted by the fact that  for each $k$ the function $s\mapsto m_{k+s}$, even though continuous on $(0,1)$, diverges as $s$ tends to $0^+$ or $1^-$. This is a consequence of results by Maz'ya--Shaposhnikova \cite{MS} (see also \cite[Section 1]{CKP}) and by Bourgain, Brezis and Mironescu \cite{BBM} (see also Ponce \cite{ponce}) which, reformulated for $k$-th derivatives, state that
the functionals 
$$
G_s(u)=s(1-s) \int_{I\times I}\frac{|u^{(k)}(x)-u^{(k)}(y)|^2}{|x-y|^{1+2s}} dx\,dy
$$
$\Gamma$-converge to $2\int_I |u^{(k)}|^2dx$ and to $\int_I |u^{(k+1)}|^2dx$ in the $L^2$-topology, as $s\to 0^+$ and $s\to 1^-$, respectively. With this remark in mind, we can correct this dependence by modifying  $F_\e$ above  for $s\in(0,1)$, setting
$$
F^{k,s}_\e(u)=\frac1\e\int_I W(u(x))dx+\frac{s(1-s)}{2^{1-s}}\e^{2(k+s)-1}\int_{I\times I}\frac{|u^{(k)}(x)-u^{(k)}(y)|^2}{|x-y|^{1+2s}} dx\,dy.
$$ 
These functionals are of the same form as $F_\e$ and their $\Gamma$-limit has the coefficient
\begin{eqnarray}\label{emmeks}\nonumber
&& m_{k}(s) =\inf\Big\{\int_{\mathbb R} W(v)dx+\frac{s(1-s)}{2^{1-s}}\int_{\mathbb R^2}\frac{|u^{(k)}(x)-u^{(k)}(y)|^2}{|x-y|^{1+2s}} dx\,dy,
 :\\
&&\hskip3cm v\in H^{k+s}_{\rm loc}(\mathbb R), \lim_{x\to\pm\infty}v(x)=\pm1\Big\}.
\end{eqnarray}
With this correction, the above-mentioned results by Maz'ya--Shaposhnikova \cite{MS} and Ponce \cite{ponce}, together with a compactness argument in which we again use the interpolation techniques, ensure that $m_{k}(s)\to m_k$ as $s\to 0^+$ and $m_{k}(s)\to m_{k+1}$ as $s\to 1^-$, respectively, where $m_k$ is defined by \eqref{mint-k} on integer $k$, thus obtaining a continuous function on $(\frac12,+\infty)$. We note that the same correction argument for $m_s$ in the particular case $k=0$ and $s\to 1^-$ can be found in \cite{PaVi}.

\section{Setting of the problem}
We preliminarily introduce some notation. We will use standard notation for  Sobolev and Lebesgue spaces, and for spaces of differentiable functions. In particular we will use $\|u\|_{L^p(I)}$ and $\|u\|_{C^k(I)}$ for the norm of $u$ in the corresponding spaces. The space $BV(I;\{-1,1\})$ is the set of piecewise-constant function on $I$ taking only the values $-1$ and $1$. For such a function $S(u)$ will denote its set of discontinuity points.
The letter $C$ will denote a generic strictly positive constant depending only on fixed parameters.

\smallskip
We will deal with singular-perturbation problems defined in fractional Sobo\-lev spaces $H^{r}$, with $r>0$ for whose definition, notation and main properties we refer to \cite{leofrac}. For ease of notation we will write $r=k+s$ with $k\in\mathbb N$ and $s\in[0,1)$. We recall that if $I$ is an interval and $s\in(0,1)$, a function $u$ belongs to $H^{k+s}(I)$ if $u$ belongs to the classical Sobolev space $H^k(I)$ and $[u]^2_{k+s}(I) <+\infty$, where
$$
[u]^2_{k+s}(I) = \int_{I\times I} \frac{|u^{(k)}(x)-u^{(k)}(y)|^2}{|x-y|^{1+2s}}dxdy.
$$
Note that $[u]^2_{k+s}(I)= [u^{(k)}]_s(I)$. If $s=0$ we set $[u]^2_k(I)=\int_I|u^{(k)}|^2dx$. 
\bigskip

Let $W\colon\mathbb R\to [0,+\infty)$ be a usual {\em double-well potential}; that is, a non-negative continuous function with $W(z)=0$ only if $z\in\{-1,1\}$.  As growth conditions are concerned, we only require that there exist $\alpha_W,\beta_W>0$ such that
\medskip 

\centerline{\hfill $W(z)\geq \alpha_W\min\{(z+1)^2,(z-1)^2, \beta_W\}$\hfill(H1) }

\medskip
\noindent  for all $z\in\mathbb R$, and that there exists $\overline \eta>0$ and $\gamma_W$ such that\medskip 

\centerline{\hfill $W(z)\le \gamma_W \min\{(z+1)^2,(z-1)^2\}$ \hfill(H2) }

\medskip\noindent  if $||z|-1|\le \overline \eta$.

Conditions (H1) and (H2) are satisfied if there exist $W''(1), W''(-1)>0$, and $W$ is larger than a positive constant at infinity.

\medskip
Our main result is the following. 
Note that the condition $k+s>\frac12$ ensures that the coefficient of the
perturbation tends to $0$ as $\e\to0$. 

\begin{theorem}\label{main} Let $k+s>\frac12$ and let $F_\e$ be defined by
$$
F_\e(u)=\frac1\e\int_I W(u(x))dx+\e^{2(k+s)-1}[u]^2_{k+s}(I).
$$
Then the $\Gamma$-limit of $F_\e$ as $\e\to 0$ with respect to the convergence in measure and in $L^1(0,1)$ is the sharp-interface functional
$F(u)= m_{k+s}\# S(u)$ defined on $BV(I;\{-1,1\})$, with
\begin{equation}\label{emme}
m_{k+s} =\inf\Big\{\int_{\mathbb R} W(v)dx+[v]^2_{k+s}(\mathbb R) : v\in H^{k+s}_{\rm loc}(\mathbb R), \lim_{x\to\pm\infty}v(x)=\pm1\Big\}.
\end{equation}
\end{theorem}

This theorem is complemented by the corresponding {\em equi-coerciveness result} Theorem \ref{compth}. 

\begin{remark}[convergence of minimum problems]\rm
If $W(z)\ge c_1|z|^p-c_2$ for some positive constants $c_1,c_2$ and $p\ge1$, then from the coerciveness in measure we deduce that the functionals $F_\e$ are equi-coercive in $L^1(I)$. As a consequence, we deduce the convergence of minimum problems of the form
$$
\min
\bigg\{F_\e(u)-\int_Ifu\,dt\bigg\}\qquad \hbox { and }\qquad\min\bigg\{F_\e(u): \int_Iu\,dt=m\bigg\},
$$
with $f\in L^{p'}(I)$ and $m\in(0,|I|)$, to the corresponding problems for $F$.  Analogously, one obtains the convergence with suitably defined boundary data (see \cite{SV}).
\end{remark}

\section{Preliminary results} 
 In this section we gather some results that will be used in the proofs. The first part contains interpolation results for fractional Sobolev spaces, the second part contains properties of sequences of functions with bounded energies $F_\e$  deriving from those interpolation results.

\subsection{Interpolation results in fractional Soboles spaces}

The following two propositions are derived from known results in the theory of fractional Sobolev spaces.

\begin{proposition}[Interpolation inequality]\label{interpolazione} 
Let $k\in \mathbb N$, $k\geq 1$ and $s\in(0,1)$. 
There exists a constant $R_{k,s}>0$ such that for all 
$I\subset \mathbb R$ bounded interval, and $u\in H^{k+s}(I)$, the interpolation inequality 
\begin{equation}\label{interpolazione-eq} 
\|u^{(\ell)}\|_{L^2(I)}\leq R_{k,s}\big(|I|^{-\ell}\|u\|_{L^2(I)}+\|u\|^\theta_{L^2(I)}[u]^{1-\theta}_{k+s}(I) \big) 
\end{equation} 
holds for all $\ell\in \{1,\dots, k\}$, where $\theta=1-\frac{\ell}{k+s}$.  
\end{proposition}

\begin{proof} Theorem 5.8 and Exercise 5.9 in \cite{leofrac} show that 
$$
\|u^{(\ell)}\|_{L^2(I)}\leq C\big(t^{-\ell}\|u\|_{L^2(I)}+t^{k-\ell}[u^{(k)}]_{s}(I) \big)
$$
for all $0<t\le |I|$.
An algebraic manipulation of this inequality (see e.g.~\cite[Exercise 1.26(ii)]{leofrac}) gives \eqref{interpolazione-eq}.
\end{proof}

\begin{lemma}[Estimate of $\|u\|_{L^2}$ with the $s$-seminorm]
\label{stimaseminorma}
There exists $C>0$ such that 
for any $I\subset \mathbb R$ open bounded interval and $u\in H^s(I)$ the following inequality holds: 
\begin{equation}\label{stimaseminormaeq} 
\|u\|_{L^2(I)}\leq C\Big(|I|^{-\frac{1}{2}}\bigg|\int_J u(x)\, dx\bigg| +|I|^s [u]_{s}(I)\Big),  
\end{equation} 
where $J\subseteq I$ is any interval such that $2|J|\geq |I|$.  
\end{lemma}

\begin{proof} 
Let $s^\prime<\min\{\frac{1}{2},s\}$ be fixed, and let $p=p(s^\prime)=\frac{2}{1-2s^\prime}$. 
Then (see \cite[Corollary 2.3]{leofrac}) there exists a positive constant $C$ such that for any $I\subset \mathbb R$ an open bounded interval and $u\in H^s(I)$ we can estimate 
\begin{equation}\label{stimaseminormacor} 
\|u\|_{L^{p}(I)}\leq C\Big(|I|^{\frac{1}{p}-1}\Big|\int_J u(x)\, dx\Big| + [u]_{s^\prime}(I)\Big),  
\end{equation} 
where $J\subseteq I$ is any interval such that $2|J|\geq |I|$.  
By H\"older's inequality we get 
$$\|u\|^2_{L^2(I)}\leq |I|^{2s^\prime}\|u\|^2_{L^{p}(I)},$$ 
so that 
\begin{eqnarray*}
\|u\|_{L^2(I)}&\leq&C|I|^{s^\prime}\bigg(|I|^{\frac{1}{p}-1}\Big|\int_J u(x)\, dx\Big| + [u]_{s^\prime}(I)\bigg)\\
&=& C\bigg(|I|^{-\frac{1}{2}}\Big|\int_J u(x)\, dx\Big| +|I|^{s^\prime} [u]_{s^\prime}(I)\bigg). 
\end{eqnarray*} 
Lemma 2.6 in \cite{leofrac} ensures that 
$[u]_{s^\prime}(I)\leq |I|^{s-s^\prime}[u]_{s}(I)$,
and the proof of the claim is complete. 
\end{proof}

\subsection{Analysis of functions with values close to the wells}
This section contains the main technical results that will be used in the sequel.
For ease of notation, we will consider functions defined on $(0,1)$, for which we drop the dependence on the interval, simply writing $[u]_{k+s}$ in the place of $[u]_{k+s}(0,1)$. All results will be independent of this choice.

The letter $I$ will be then used for a generic interval $I\subset \mathbb R$, for which we set
$$
F_\e(u;I)=\frac1\e\int_I W(u(x))dx+\e^{2(k+s)-1}[u]^2_{k+s}(I).
$$

\smallskip
Let $\{u_\e\}$ be a sequence of functions in $H^{k+s}(0,1)$ such that 
\begin{equation}\label{defS}
\sup_{\e>0}F_\e(u_\e)=S<+\infty. 
\end{equation}
We can suppose that $u_\e\in H^{k+s}(0,1)\cap C^\infty([0,1])$ by the density of such functions with respect to the convergence in $H^{k+s}(0,1)$ (see e.g.~\cite[Theorem 1.64]{leofrac}). 

\smallskip
With fixed $\eta\in(0,\min\{1,\sqrt{\beta_W}\})$ we set 
\begin{equation}\label{defAe}
A_\e^\eta=\{t\in(0,1): \big||u_\e(t)|-1\big|<\eta\}. 
\end{equation}
Hypothesis (H1) implies that we have 
$$W(z)\geq \alpha_W(z-1)^2 \ \hbox{\rm in }[1-\eta,1+\eta] \ \  \hbox{\rm and} \ \ W(z)\geq \alpha_W(z+1)^2 \ \hbox{\rm in }[-1-\eta,-1+\eta]$$
for all $\eta\in (0,\min\{1,\sqrt{\beta_W}\})$.
Hence also
\begin{equation}\label{mis} 
|\{t\in(0,1): ||u_\e(t)|-1|\geq \eta \}|\leq \frac{S\e }{\alpha_W}\Big(\frac{1}{\eta^2}+\frac{1}{\beta_W}\Big)<\frac{2S\e }{\alpha_W\eta^2}.
\end{equation}

\begin{lemma}[Bound on the derivatives]\label{stimaintervallo} 
Let $k\neq 0$ and $\eta\in(0,\min\{1,\sqrt{\beta_W}\})$ be fixed. 
Then there exists a constant $R=\widehat R_{k,s}>0$ such that for any interval $I\subset A_\e^\eta$ with $|I|\geq R\e$ we have 
$$\Big|\Big\{t\in I: |u_\e^{(\ell)}(t)|<\frac{1}{\e^{\ell}} \ \hbox{\rm for all } \ell\in\{1,\dots, k\}\Big\}\Big|>0.$$
\end{lemma}
\begin{proof} 
It is not restrictive to assume that $I\subset A_\e^\eta$ is such that  $|u_\e(t)-1|<\eta$ for all $t\in I$, the case  $|u_\e(t)+1|<\eta$ being completely analogous. Then, hypothesis (H1) on $W$ ensures that 
\begin{equation}\label{l2w}
\|u_\e-1\|^2_{L^2(I)}\leq \frac{1}{\alpha_W}\int_I W(u_\e)\, dt. 
\end{equation}
Fixing  $p=\frac{1}{\theta}=\frac{k+s}{k+s-\ell}$ and $q=\frac{p}{p-1}=\frac{k+s}{\ell}$, we can apply the convexity inequality 
$AB\leq \frac{1}{p}A^p+\frac{1}{q}B^q$ with the choice 
$$A=\e^{-\theta} \|u_\e-1\|^{2\theta}_{L^2(I)}, \quad B=\e^{2\ell-1+\theta} [u_\e]_{k+s}^{2(1-\theta)}(I),$$ 
obtaining 
\begin{equation}\label{stconv}
\e^{2\ell-1}\|u_\e-1\|^{2\theta}_{L^2(I)}[u_\e]^{2(1-\theta)}_{k+s}(I)\leq \frac{1}{\e}\|u_\e-1\|^2_{L^2(I)}+\e^{2(k+s)-1}[u_\e]^2_{k+s}(I).
\end{equation} 

By Proposition \ref{interpolazione} applied with $u=u_\e-1$, estimates \eqref{l2w} and assumption \eqref{defS}, this inequality allows us to deduce that 
\begin{eqnarray}\label{interpolationue}\nonumber
\|u_\e^{(\ell)}\|^2_{L^2(I)}&\leq &
R_{k,s}^2\Big(|I|^{-\ell}\|u_\e-1\|_{L^2(I)}+\|u_\e-1\|^\theta_{L^2(I)}[u_\e]^{1-\theta}_{k+s} (I)\Big)^2 
\\ \nonumber
&\leq &
2R_{k,s}^2\Big(|I|^{-2\ell}\|u_\e-1\|^2_{L^2(I)}+\|u_\e-1\|^{2\theta}_{L^2(I)}[u_\e]^{2(1-\theta)}_{k+s}(I) \Big) 
\\
&\le &2R^2_{k,s} \max\{1,\alpha_W^{-1}\} S (|I|^{-2\ell}+
\e^{-2\ell}) \e 
\end{eqnarray}
for all $\ell\in\{1,\dots, k\}$, where $R_{k,s}$ is the constant in \eqref{interpolazione-eq}. 

We now suppose that 
\begin{equation}\label{contr}
\Big|\Big\{t\in I: |u_\e^{(\ell)}(t)|<\frac{1}{\e^{\ell}} \ \hbox{\rm for all } \ell\in\{1,\dots, k\}\Big\}\Big|=0.  
\end{equation} 
Then, there exists $\ell\in \{1,\dots, k\}$ such that 
$$\Big|\Big\{t\in I: |u_\e^{(\ell)}(t)|\geq\frac{1}{\e^{\ell}}\Big\} \Big|\geq \frac{|I|}{k},$$ 
so that by \eqref{interpolationue} we get 
 \begin{equation*}
\frac{|I|}{k} \frac{1}{\e^{2\ell}}\leq 2R^2_{k,s} \max\{1,\alpha_W^{-1}\} S (|I|^{-2\ell}+
\e^{-2\ell})  \e,  
\end{equation*}
which implies that either 
\begin{equation}\label{est1}
|I|^{1+2\ell} \leq 4R^2_{k,s} k \max\{1,\alpha_W^{-1}\} S\e^{1+2\ell}
\end{equation}  
or 
\begin{equation}\label{est2}
|I|\leq 4R^2_{k,s} k 
\max\{1,\alpha_W^{-1}\} S\e.
\end{equation}

We have then shown that if \eqref{contr} holds, there exists a positive constant $R=\widehat R_{k,s}$ independent of $\eta$, $I$ and $\e$ such that $|I|<R\e$, proving the claim.  
\end{proof} 

\begin{remark}\label{dernulla}\rm Let $k\neq 0$, and 
let $I$ be an open interval containing $2k+3$ disjoint maximal intervals $(c_n,d_n)$ such that $(c_n,d_n)\subset A_\e^\eta$ for each $n$. Then, for any $\ell\in\{1,\dots, k\}$ there exists a point $t_\e^\ell\in I$ with $u_\e^{(\ell)}(t_\e^\ell)=0$. 
Indeed, in each interval $(c_n,c_{n+2})$ there exists at least one critical point for $u_\e$; that is, there exist $k+1$ points in $I$ such that $u_\e^\prime$ vanishes. The claim follows by an iterated use of the Mean Value Theorem (or Rolle's Theorem). 
\end{remark}

\begin{lemma}[Bound of the length of intervals with ``large'' derivatives]\label{datibordo} 
Let $k\neq 0$ and $\eta\in(0,\min\{1,\sqrt{\beta_W}\})$ 
be fixed. Let $\widehat R_{k,s}$ be the positive constant given by Lemma {\rm\ref{stimaintervallo}}, and let 
\begin{equation}\label{defL}
L=L_{k,s}(\eta, S)=\frac{2S}{\alpha_W\eta^2}+(2k+3)\widehat R_{k,s}.
\end{equation} 
Then, for all $\e_r\to 0$ as $r\to+\infty$ and $\{I_r\}$ sequence of intervals satisfying 

\smallskip 

{\rm(i)} $|I_r|\geq 4L\e_r$ for any $r$,

{\rm(ii)} $\lim\limits_{r\to+\infty}F_{\e_r}(u_{\e_r};I_r)=0$

\smallskip 

\noindent there exists $r_0\in \mathbb N$ such that for all $r\geq r_0$ 
$$\Big|\Big\{t\in I_r\cap A_{\e^r}^\eta: |u^{(\ell)}_{\e_r}(t)|<\frac{1}{\e_r^\ell}\ \hbox{\rm for all }\ \ell\in\{1,\dots, k\}\Big\}\Big|>0.$$
\end{lemma}

\begin{proof} 
Let $J_r=(\alpha_r,\beta_r)\subset I_r$ be such that $|J_r|=4L\e_r$.  
We argue by contradiction, and, in view of Lemma \ref{stimaintervallo}, we suppose that (up to subsequences) $J_r$ does not contain a subinterval of $\{t:||u_{\e_r}|-1|<\eta\}$ with length strictly greater than $\widehat R_{k,s}\e_r$. 
Since the measure of the set 
$\{t:||u_{\e_r}|-1|\geq\eta \}$ is estimated by \eqref{mis},  the definition of $L$ and the assumption by contradiction imply that
both the intervals $(\alpha_r, \alpha_r+L\e_r)$ and $(\beta_r-L\e_r, \beta_r)$ contain $2k+3$ disjoint maximal intervals where $||u_{\e_r}|-1|<\eta$. 
If $k\geq 2,$ by Remark~\ref{dernulla} we deduce that there exist $\alpha^k_r\in(\alpha_r, \alpha_r+L\e_r)$
and $\beta^k_r\in (\beta_r-L\e_r, \beta_r)$ such that $u^{(k-1)}_{\e_r}(\alpha^k_r)=u^{(k-1)}_{\e_r}(\beta^k_r)=0$, while for $k=1$ we simply choose
$\alpha^1_r\in(\alpha_r, \alpha_r+L\e_r)$
and $\beta^1_r\in (\beta_r-L\e_r, \beta_r)$ in the set of the endpoints of the maximal intervals in $A_{\e_r}^\eta$.

 Note that for all $k$ we have that $\beta^k_r-\alpha^k_r\geq 2L\e_r=\frac{|J_r|}{2}$. 
Moreover, again by Remark \ref{dernulla}, for any $\ell\in\{1,\dots, k\}$ there exist $t^\ell_r\in J_r$ such that $u^{(\ell)}_{\e_r}(t^\ell_r)=0$. Let $t^0_r\in J_r$ be a point such that $||u_{\e_r}(t^0_r)|-1|=\eta$.  
Now, we set $v_r(x)=u_{\e_r}(\alpha_r+4L\e_r x)$, so that  
$v^{(\ell)}_r(x_r^\ell)=0$ and $||v_r(x_r^0)|-1|=\eta$ with $x_r^\ell=\frac{t^\ell_r-\alpha_r}{4L\e_r}$. An iterated use of the Fundamental Theorem of Calculus then gives the estimate   
\begin{equation}\label{stimahk}
\|v_r\|_{H^k(0,1)}\leq C(L) \big(1+\|v_r^{(k)}\|_{L^2(0,1)}\big).
\end{equation} 
Setting $y_r=\frac{\alpha^k_r-\alpha_r}{4L\e_r}$ and 
$z_r=\frac{\beta^k_r-\alpha_r}{4L\e_r}$, since $|z_r-y_r|\geq \frac{1}{2}$, we can apply Lemma~\ref{stimaseminorma} to the function $v_r^{(k)}$, with $I=(0,1)$ and  $J=(y_r,z_r)$, obtaining 
$$\|v_r^{(k)}\|_{L^2(0,1)}\leq C\bigg(\Big|\int_{y_r}^{z_r} v_r^{(k)}(x)\, dx\Big| + [v_r]_{k+s}\bigg).$$ 
Since by construction $v^{(k-1)}_r(y_r)=v^{(k-1)}_r(z_r)=0$, and, in the case $k=1$, $|v_r(y_r)|$ and $|v_r(z_r)|$ are bounded by $1+\eta$, recalling \eqref{stimahk} it follows that 
$$\|v_r\|_{H^k(0,1)}\leq C (1+[v_r]_{k+s}).$$
where $C$ is a possibly different positive constant, independent of $r$. 
Changing variable we have
$$[v_r]^2_{k+s}=(4L\e_r)^{2(k+s)-1}[u_r]^2_{k+s}(J_r),$$
so that 
$$\|v_r\|_{H^k(0,1)}\leq C \big(1+ (4L)^{2(k+s)-1}F_{\e_r}(u_{\e_r};J_r)\big).$$ 
Since $F_{\e_r}(u_{\e_r};J_r)\to 0$ as $r\to+\infty$,  
we then deduce 
the equiboundedness of $v_r$ in $H^k$, so that, up to subsequences,  there exists $v\in H^k$ such that $v_r\to v$ uniformly. Hence, 
\begin{eqnarray*}
&&0\ =\lim_{r\to+\infty}F_{\e_r}(u_{\e_r};I_r)\ \geq\ \liminf_{r\to+\infty}\frac{1}{\e_r}\int_{J_r} W(u_{\e_r})\, dx\\ &&\ \ \ =\ 4L\liminf_{r\to+\infty}\int_0^1 W(v_r)\, dx\ 
\geq\  4L\int_0^1 W(v)\, dx.
\end{eqnarray*}
This implies that either $v=-1$ in $(0,1)$, or $v=1$ in $(0,1)$, which is a contradiction since 
for all $r$ we have $||v_r(x_r^0)|-1|=\eta\neq 0$. 
\end{proof}

\begin{remark}\rm
The claim of Lemma \ref{datibordo}  holds without requiring condition (ii) if we assume that  
$$\lim_{r\to+\infty}\frac{|I_{r}|}{\e_r}=+\infty.$$ 
Indeed, we set $N_r=\lfloor \frac{|I_r|}{4L\e_r}\rfloor,$ and consider a family of $N_r$ disjoint subintervals of $I_r$ with length $4L\e_r$. Then, there exists an interval $J_r=(a_r,b_r)$ such that $|J_r|=4L\e_r$ such that 
$$F_{\e_r}(u_{\e_r};J_r)\leq \frac{S}{N_r},$$ and we can apply Lemma \ref{datibordo}.  
\end{remark}

\section{Compactness} 
In this section we prove that the sequence $F_\e$ is equicoercive; more precisely, that sequences $u_\e$ with equibounded energies are pre-compact in measure. We recall that we consider functions defined on $(0,1)$ for ease of notation.

\begin{theorem}[Compactness]\label{compth} 
Let $\{u_\e\}$ be a family in $H^{k+s}(0,1)$ such that $\sup_{\e>0} F_\e(u_\e)\le S<+\infty$. Then, there exist $u\in BV((0,1);\{-1,1\})$ and a subsequence $\e_r\to 0$ such that $u_{\e_r}\to u$ in measure.  Furthermore, if  $W$ satisfies $W(z)\ge c_1|z|^p-c_2$ for some $c_1,c_2$ and $p\ge 1$, then $u_{\e_r}\to u$ in $L^p(0,1)$.
\end{theorem}

Before proving Theorem \ref{compth}, we give the definition of transition intervals (between two values $\lambda_1<\lambda_2$) of the sequence $u_\e$. Let $v$ be a continuous function; we say that an open interval $I\subset(0,1)$ is a {\em transition interval for $v$ between $\lambda_1$ and $\lambda_2$} if $v(t)\in(\lambda_1,\lambda_2)$ for $t\in I$, and $\{v(\inf I), v(\sup I)\}=\{\lambda_1,\lambda_2\}$. Note that different transition intervals are disjoint.
With given a sequence $u_\e$, the set of the transition intervals for $u_\e$ between $\lambda_1$ and $\lambda_2$ is denoted by $\mathcal I_\e(\lambda_1,\lambda_2)$. 

We will prove a bound on the number of transition intervals ``not crossing the wells''; that is, between two values $\lambda_1<\lambda_2$ such that $[\lambda_1,\lambda_2]\cap \{-1,1\}=\emptyset$. In particular, with fixed
$\eta\in(0,\min\{1,\sqrt{\beta_W}\})$, we will consider transition intervals between $-1+\eta$ and $1-\eta$, which will be called $\eta${\em-transition intervals}.
With given a sequence $u_\e$, the set of the $\eta$-transition intervals for $u_\e$ is denoted by $\mathcal I_\e^\eta$.

In the following result we prove that the number of transitions between two values as above is equibounded. In particular, we will prove the compactness theorem by using the fact the number of the $\eta$-transitions is equibounded, while the general case will be used in the proof of the lower-bound estimate. 

\begin{proposition}[Bound on the number of transitions]\label{numbound}  
Let $k+s>\frac12$ and let $\{u_\e\}\subset H^{k+s}(0,1)\cap C^\infty([0,1])$ with $F_\e(u_\e)\le S<+\infty$. Let $\lambda_1<\lambda_2$ be such that 
$[\lambda_1,\lambda_2]\cap\{-1,1\}=\emptyset$, and let $\e_r\to 0$. Then,
$\sup_r \#\mathcal I_{\e_r}(\lambda_1,\lambda_2)<+\infty$.
\end{proposition}

\begin{proof} 
We first treat the case $k=0$ and $s>\frac12$ as it uses a simpler argument. 
Suppose by contradiction that there exist $M_r$ transition intervals for $u_r:=u_{\e_r}$ with $M_r\to+\infty$. Then there exists a transition interval $I_r$ between $\lambda_1$ and $\lambda_2$ such that
\begin{equation}\label{lowen2.1}
F_{\e_r}(u_r; I_r)=\frac{1}{\e_r}\int_{I_r}W(u_r)\, dx+\e_r^{2s-1}[u_r]^2_{k+s}(I_r)\leq \frac{S}{M_r}=o(1).
\end{equation}
Since $[\lambda_1,\lambda_2]\cap\{-1,1\}=\emptyset$ we have $W(u_r)\ge C>0$ on $I_r$, with $C$ depending only on $\lambda_1$ and $\lambda_2$, so that in particular
$$
\frac{|I_r|}{\e_r}\le \frac1{C\e_r}\int_{I_r} W(u_r)\,dx\le \frac1CF_{\e_r}(u_r; I_r)=o(1).
$$
If $I_r=(\alpha_r,\beta_r)$, we set $v_r(t)= u_r(\alpha_r+(\beta_r-\alpha_r)t)-u_r(\alpha_r)$. A change of variables gives
$$
[v_r]^2_s(0,1)= (\beta_r-\alpha_r)^{2s-1}[u_r]^2_s (I_r)\le \Big(\frac{|I_r|}{\e_r}\Big)^{2s-1}F_{\e_r}(u_r; I_r)=o(1).
$$
The compact embedding of $H^s(0,1)$ in $C^{0,s'}(0,1)$ for $s'\le s-\frac12$ and the fact that $v_r(0)=0$ imply the uniform convergence of $v_r$ to $0$, which gives a contradiction since $|v_r(1)|=\lambda_2-\lambda_1$.

\medskip
If $k\geq 1$ an argument using the compact embedding of $H^{k+s}$  is not immediately possible. 
In this case, we suppose by contradiction that there exist $\widehat M_r\to+\infty $,  and
$(a_r^j,b_r^j)\in \mathcal I_{\e_r}(\lambda_1,\lambda_2)$ for $j\in\{1,\ldots, \widehat M_r\}$ transition intervals such that $a_r^j<a_r^{j+1}$ for all $r$ and for all $j\in\{1,\dots, \widehat M_r^\eta-1\}$.  We set
$$
M_r=\bigg\lfloor \frac{\widehat M_r}{6(k+2)}\bigg\rfloor,
$$
and define the intervals 
$$I_r^{n}=(a_r^{1+6n(k+2)}, b_r^{6(n+1)(k+2)}), \ \ n\in\big\{0,\dots, M_r\big\}.$$ 
By construction each $I_r^{n}$ contains at least $6(k+2)$ transition intervals for $u_r$. 

We then choose $n_r\in \{0,\dots, M_r\}$ such that, setting $I_r=I_r^{n_r}$, 
\begin{equation}\label{lowen2}
F_{\e_r}(u_r; I_r)=\frac{1}{\e_r}\int_{I_r}W(u_r)\, dx+\e_r^{2(k+s)-1}[u_r]^2_{k+s}(I_r)\leq \frac{S}{M_r}.
\end{equation} 

The argument by contradiction in the following will consist in a scaling argument producing a sequence of functions with equibounded $H^{k+s}$ norm on a common interval and uniformly converging to a constant. Since the same functions are transition functions this will prove the contradiction. In order to bound the $H^{k+s}$ norm we make an iterated use of the Fundamental Theorem of Calculus, which will be possible since for each $\ell\in\{1,\ldots,k-1\}$ we will find points of the reference intervals where the $\ell$-th derivative vanishes. This in turn is possible by the simple observation that in an interval with $k$ transitions there are at least $k-1$ extremal points, and hence $k-1$ points where the first derivative vanishes. Repeating the argument, the $\ell$-th derivative vanishes in at least $k-\ell$ points.

We note that this argument can be directly applied to the interval $I_r$ if it is of length  
of order $\e_r$, while otherwise it will be applied to suitably chosen subintervals by use of interpolation arguments (for which a higher number of transition intervals is needed, which explains the number $6(k+2)$ above).
We separately treat the two cases.
\medskip \goodbreak

\noindent {\em Case $1$: $\sup_{r}\frac{|I_r|}{\e_r}\leq T<+\infty$.} 

\smallskip
We divide $I_r$ in three subintervals, each one containing (at least) $2k+3$ transition intervals between $\lambda_1$ and $\lambda_2$; we set 
$$I^-_r=(a_r^{1+6n_r(k+2)}, b_r^{1+(6n_r+2)(k+2)}),\ I^+_r=(a_r^{1+(6n_r+4)(k+2)}, b_r^{(6n_r+6)(k+2)}).$$ 
Since both $I^-_r$ and $I^+_r$ contain at least $k$ transition intervals between $\lambda_1$ and $\lambda_2$, 
there exist $\alpha^1_r\in I^-_r$ and $\beta^1_r\in I^+_r$ such that 
$u_r(\alpha^1_r)=u_r(\beta^1_r)$.
Moreover both $I^-_r$ and $I^+_r$ contain at least $k$ critical points for $u_{r}$, and, if $k\ge 2$, there exist $\alpha^k_r\in I^-_r$ and $\beta^k_r\in I^+_r$ such that 
$u_r^{(k-1)}(\alpha^k_r)=u_r^{(k-1)}(\beta^k_r)=0$. Hence, for all $k\ge 1$ in particular $u_r^{(k-1)}(\alpha^k_r)-u_r^{(k-1)}(\beta^k_r)=0$.

Setting $v_r(x)=u_r((\beta^k_r-\alpha^k_r) x+\alpha^k_r)-\lambda_1$, we get 
\begin{eqnarray}\label{stisem}
[v_r]^2_{k+s}(0,1)&=&(\beta^k_r-\alpha^k_r)^{2(k+s)-1}[u_r]^2_{k+s}(\alpha^k_r,\beta^k_r)\nonumber\\ 
&\leq&\Big(\frac{\beta^k_r-\alpha^k_r}{\e_r}\Big)^{2(k+s)-1}\frac{S}{M_r}\leq \frac{{T}^{2(k+s)-1}S}{M_r}, 
\end{eqnarray}
since \eqref{lowen2} holds. 
Noting that by construction $(\alpha^k_r,\beta^k_r)$ contains at least 
$k$ transition intervals between $\lambda_1$ and $\lambda_2$, 
we get that for any $\ell\in\{1,\dots, k\}$ there exists
$t_r^\ell\in (\alpha^k_r,\beta^k_r)$ such that $u_r^{(\ell)}(t_r^\ell)=0$, and hence $x_r^\ell\in (0,1)$ such that $v_r^{(\ell)}(x_r^\ell)=0$. Moreover, 
there exists 
$t_r^0\in (\alpha^k_r,\beta^k_r)$ with $u_r(t_r^0)=\lambda_1$; that is, there exists $x_r^0\in (0,1)$ such that $v_r(x_r^0)=0$. 
An iterated application of the Fundamental Theorem of Calculus then ensures that 
$$\|v_r\|^2_{H^k(0,1)}\leq C(T) \|v_r^{(k)}\|_{L^2(0,1)}^2.$$
We can apply Lemma \ref{stimaseminorma} with $u=v_r^{(k)}$ and $I=J=(0,1)$, so that $\int_J v_r(t)dt =v_r^{(k-1)}(1)-v_r^{(k-1)}(0)=0$. It follows that 
$$\|v_r^{(k)}\|_{L^2(0,1)}^2\leq C [v_r]^2_{k+s}(0,1).$$ 
Since $M_r\to+\infty$, the term in the right-hand side is infinitesimal by \eqref{stisem}, 
so that $v_r$ converges to $0$ in ${H^k(0,1)}$. We then deduce that $v_r$ uniformly converges to $0$. This gives a contradiction, since for all $r$ there exist points $x_r\in(0,1)$ such that $v_r(x_r)=\lambda_2-\lambda_1$. 
\medskip \goodbreak

\noindent {\em Case $2$: $\sup_{r}\frac{|I_r|}{\e_r}=+\infty$.} 

\smallskip
We can suppose, up to subsequences, that $\frac{|I_r|}{\e_r}$ diverges. 
Hence, for $r$ large enough, 
we can choose an interval $J_r=(\alpha^\prime_r, \alpha^\prime_r+4L\e_r)\subset I_r$, where $L=L_{k,s}(\eta, S)$ is the positive constant defined in \eqref{defL}, containing at least a transition interval for $u_r$ between $\lambda_1$ and $\lambda_2$.  
Then, since \eqref{lowen2} holds, and $M_r\to+\infty$, we can apply Lemma \ref{datibordo} and deduce that for $r$ large enough and for all $\ell\in\{1,\dots, k\}$ there exists $t^\ell_r\in J_r$ such that 
$|u_r^{(\ell)}(t^\ell_r)|<\frac{1}{\e_r^{\ell}}$; moreover, since $J_r$ contains a transition interval for $u_r$, there exists $t_r^0\in J_r$ such that   $u_r(t_r^0)=\lambda_2$. 

Now,  either $(\alpha^\prime_r, \alpha^\prime_r+4L\e_r)$ contains $2k+3$ disjoint maximal subintervals of $A_{\e_r}^\eta$, or contains a subinterval of $A_{\e_r}^\eta$ with length strictly greater than $\widehat R_{k,s}\e_r$. The same holds for $(\alpha^\prime_r+3L\e_r, \alpha^\prime_r+4L\e_r)$. Hence, by Remark \ref{dernulla} in the first case, or by Lemma \ref{stimaintervallo} otherwise, there exist $t_r^-\in (\alpha^\prime_r, \alpha^\prime_r+L\e_r)$ and $t_r^+\in (\alpha^\prime_r+3L\e_r, \alpha^\prime_r+4L\e_r)$ such that $|u_r^{(k-1)}(t_r^-)|<\frac{1}{\e_r^{k-1}}$ and $|u_r^{(k-1)}(t_r^+)|<\frac{1}{\e_r^{k-1}}$. 

Setting $v_r(x)=u_r(\alpha_r^\prime+4L\e_r x)-\lambda_1$, recalling \eqref{lowen2} we obtain the estimate 
\begin{equation}\label{stisem2}
[v_r]^2_{k+s}(0,1)
\leq (4L)^{2(k+s)-1}\frac{S}{M_r}. 
\end{equation} 
By the Fundamental Theorem of Calculus it follows that 
$$\|v_r\|^2_{H^k(0,1)}\leq C(L) \|v_r^{(k)}\|_{L^2(0,1)}^2.$$ 
We can then apply Lemma \ref{stimaseminorma} with $u=v_r^{(k)}$, $I=(0,1)$ and  $J=(x_r^-,x_r^+)$, 
where  
$x_r^-=\frac{t_r^--\alpha_r^\prime}{4L\e_r}$ and 
$x_r^+=\frac{t_r^+-\alpha_n^\prime}{4L\e_r}$,  
since $x_r^+-x_r^-\geq \frac{1}{2}$.  
Recalling that both $|v_r^{(k-1)}(x_r^+)|$ and $|v_r^{(k-1)}(x_r^-)|$ are not greater than $1$, it follows that  
$$\|v_r^{(k)}\|_{L^2(0,1)}^2\leq C \Big(1+[v_r]^2_{k+s}(0,1)\Big),$$
and $v_r$ is then equibounded in $H^k(0,1)$. 
Up to subsequences, $\{v_r\}$ uniformly converges to $v\in H^k(0,1)$;  
then, recalling that 
\begin{equation}\label{stiW}
\int_{0}^1 W(v_r+\lambda_1)\, dx=\frac{1}{4L\e_r}\int_{J_r} W(u_r)\, dx\leq \frac{1}{4L}\frac{S}{M_r},  
\end{equation}
we obtain 
\begin{equation}\label{assurdo2}
0=\liminf_{r\to+\infty}\int_{0}^1 W(v_r+\lambda_1)\, dx\geq \int_{0}^1 W(v+\lambda_1)\, dx,  
\end{equation}
which is a contradiction since $v+\lambda_1$ cannot be constantly equal to $1$ or $-1$, since for all $r$ there exist points where $v_r=0$ and  points where $v_r=\lambda_2-\lambda_1$, corresponding to the endpoints of a transition interval for $u_r$.
\end{proof}

\begin{remark}\rm In the following proof we use Proposition \ref{numbound} applied to $\eta$-transitions. Note that in the following it will also be used in the proof of the lower bound to estimate the number of transitions e.g.~between $1+\eta$ and $1+2\eta$.
\end{remark}

\begin{proof}[Proof of Theorem {\rm\ref{compth}}] 
Let $\{u_\e\}$ be such that $u_\e\in H^{k+s}(0,1)$ and 
$\sup_\e F_\e(u_\e)=S<+\infty$. As already noted, we can suppose that $u_\e\in H^{k+s}(0,1)\cap C^\infty([0,1])$ by the density of such functions with respect to the convergence in $H^{k+s}(0,1)$ (see e.g.~\cite[Theorem 1.64]{leofrac}).

Let $\eta\in(0,\min\{1,\sqrt{\beta_W}\})$, and $\e_r\to 0$. 
We can apply Proposition \ref{numbound}  with the choice $\lambda_1=-1+\eta$ and $\lambda_2=1-\eta$, obtaining that the number $\#\mathcal I_{\e_r}^{\eta}$ of the $\eta$-transition intervals is equibounded by a constant $M_\eta$ depending only on $\eta$. Since  for all $r$  we have $\#\mathcal I_{\e_r}^{\eta}\leq \#\mathcal I_{\e_r}^{\eta^\prime}$ if $\eta<\eta^\prime$,  
we obtain that 
$$\# \{\eta\hbox{\rm-transition interval for } u_r\}\leq M_0=M_{\eta_0}$$
for all $r$ and $\eta<\eta_0:=\min\{1,\sqrt{\beta_W}\}$. 
The complementary set of the union of the $\eta$-transition intervals for $u_r$ is then given by the union of an equibounded number of maximal intervals where $u_r(t)>-1+\eta$, denoted by $A^\eta_r(1)$, and the union of an equibounded number of maximal intervals where $u_r(t)<1-\eta$, denoted by $A^\eta_r(-1)$. 
We set 
$$u_{r}^\eta(t)=\begin{cases} 
1 & \hbox{\rm if } t\in A_r^\eta(1)\\
-1 & \hbox{\rm if } t\in (0,1)\setminus A_r^\eta(1). 
\end{cases}$$
Note that the discontinuity points of $u_r^\eta$ are endpoints of $\eta$-transition intervals, and hence they are bounded by $M_0$. Up to subsequences then $u_r^\eta$ converges as $r\to +\infty$ to some $u\colon(0,1)\to \{-1,1\}$ with at most $M_0$ discontinuity points. Noting that
\begin{eqnarray*}
&&\hskip-1.5cm\int_{\{|u_r|<2\}}|u_r-u_r^\eta|^2dt\\
&=& \int_{A_r^\eta(1)\cap\{ |u_r|<2\}}|u_r-1|^2dt+
\int_{A_r^\eta(-1)\cap \{|u_r|<2\}}|u_r+1|^2dt\\
&\le& \frac1{\alpha_W}\int_0^1 W(u_r)dt\le \e_r S
\end{eqnarray*}
and that $|\{|u_r|\ge 2\}|\to 0$ as $r\to+\infty$ we conclude that $u_r$ converges to $u$ in measure.

\smallskip
Finally, we prove the the last claim. To this end, note that there exists $\overline z$ such that $W(z)\ge \frac{c_1}2 |z|^p$ if $|z|\ge \overline z$, so that 
$$
\int_{\{t: |u_\e(t)|\ge \overline z\}}|u_\e|^pdx\le \frac{2\e}{c_1} S.
$$
From this, the convergence in measure of $u_\e$ implies the convergence in $L^p$.
\end{proof}

\section{Lower bound} 
The proof of Theorem \ref{main} will be subdivided into a lower and an upper bound. Here we begin by proving the lower bound. To that end, we define the auxiliary interfacial constant
\begin{eqnarray}\label{emmetilde}\nonumber
&&\widetilde m_{k+s}=\inf_{T>0} \inf\bigg\{\int_{-\infty}^{+\infty} W(v)dt + [v]^2_{k+s}(\mathbb R): v\in H^{k+s}_{\rm loc} (\mathbb R), \\
&&\hskip3cm v(x)=-1\hbox{ \rm if }x\le -T,v(x)=1\hbox{ \rm if }x\ge T\Big\}.
\end{eqnarray}

\begin{proposition}\label{lo-bo} Let $u\in BV((0,1);\{-1,1\})$, let $S(u)$ be its set of discontinuity points, and let $u_\e\to u$ in measure. Then we have
\begin{equation}
\liminf_{\e\to 0} F_\e(u_\e)\ge \widetilde m_{k+s}\,\#(S(u)),
\end{equation}
where $\widetilde m_{k+s}$ is given by \eqref{emmetilde}.
\end{proposition}

\begin{proof}
By density we can assume that $u_\e\in C^\infty([0,1])$. Moreover, it suffices to consider the case  $S=\sup_\e F_\e(u_\e)<+\infty$.  We fix a subsequence $\e_r$ tending to $0$ such that there exists the limit $$\lim_{r\to+\infty} F_{\e_r}(u_{\e_r})=
\liminf_{\e\to 0} F_{\e}(u_{\e}).
$$
For the sake of notation we relabel this subsequence again by $u_\e$.
We now prove a lower bound for the limit above. This will be done by constructing suitable test functions for $\widetilde m_{k+s}$ whose energy is not greater than that of $u_\e$ close to its jump points. 

Let $\eta>0$ be fixed. We preliminarily note that Proposition \ref{numbound} ensures that the number of the transitions of $u_\e$ between $1+\eta$ and $1+2\eta$, between $-1-2\eta$ and $-1+\eta$ and between $-1+\eta$ and $1-\eta$ is equibounded by a constant $M(\eta)$; that is, 
\begin{equation}\label{boundtrans} 
\max\big\{\#\mathcal I_\e(1+\eta, 1+2\eta), \#\mathcal I_\e(-1-2\eta, -1-\eta), \#\mathcal I_\e(-1+\eta, 1-\eta)\big\}\leq M(\eta).
\end{equation}  

Let 
$\{x_1,\dots, x_M\}$ denote the set of the jump points of $u$, and let  
$T^\prime>0$ be such that for all $j\in\{1,\dots, M\}$ the intervals $[x_j-T^\prime,x_j+T^\prime]$ are pairwise disjoint.  

We now fix $j\in\{1,\dots, M\}$ and consider the set $[x_j-T^\prime,x_j+T^\prime]$. For the sake of simplicity of notation, we apply a translation and assume that $x_j=0$; furthermore,  it is not restrictive to assume $u=-1$ in $[-T^\prime,0)$, $u=1$ in $(0,T^\prime]$. 
Since the number of the transitions as above is equibounded by \eqref{boundtrans}, and the length of each transition interval is at most of order $\e$, up to subsequences we can assume that there exists a finite number of limit points $\{y_n\}$ of the transition intervals; we fix $\tau>0$ such that $3\tau<\min\{|y_n|: y_n\neq 0\}$. 

For each fixed $\sigma\in (0,\tau)$ (which will be chosen small in the following) there exists $\e(\eta, \sigma)\in(0,1)$ such that the sets $(-3\tau,-\sigma)$ and $(\sigma,3\tau)$ do not contain transition intervals for all the functions $u_\e$ with $\e<\e(\eta, \sigma)$. 
Since $u_\e\to u$ in measure, it follows that 
\begin{eqnarray}\label{ur1}
&u_\e(t) \in (-1-2\eta, -1+2\eta) \ \ &\hbox{\rm if } \ t\in (-3\tau,-\sigma)\nonumber \\
&u_\e(t) \in (1-2\eta, 1+2\eta) \ \ &\hbox{\rm if } \ t\in (\sigma,3\tau)
\end{eqnarray} 
for all $\e\in (0,\e(\eta,\sigma))$. 
If $k\geq 1$, for all $\sigma$ there exist $\sigma^-\in (-2\sigma,-\sigma)$ and $\sigma^+\in (\sigma,2\sigma)$ such that 
\begin{equation}\label{dtj}
\max\{|u_\e^{(\ell)}(\sigma^-)|, |u_\e^{(\ell)}(\sigma^+)|\}<\frac{1}{\e^{\ell}} 
\end{equation}
for all $\ell\in\{1,\dots, k\}$. Indeed, both intervals $(-2\sigma,-\sigma)$ and $(\sigma,2\sigma)$ satisfy the 
hypotheses of Proposition \ref{stimaintervallo}. 
Condition \eqref{dtj} will be used in the following to construct extensions to the constant values $-1$ and $1$ on the whole $\mathbb R$.

Let $N\in\mathbb N$ be fixed. 
We divide $(\tau,2\tau)$ and $(-2\tau,-\tau)$ in $N$ intervals with equal length; 
then, there exists $n^+_\e$ and $n_\e^-$ in $\{1,\dots, N\}$ such that, 
letting $I^+_\e$ and $I^-_\e$ denote the corresponding interval, respectively, 
\begin{equation}\label{enb}
\frac{1}{\e}\int_{I_\e^+}W(u_\e)\, dx+\e^{2(k+s)-1}\int_{-3\tau}^{3\tau}\int_{I_\e+}\frac{|u_\e(x)-u_\e(y)|^2}{|x-y|^{2s+1}}\,dx\,dy\leq \frac{S}{N}, 
\end{equation}
and the same holds for $I_\e^-$. 
Let $x^+_\e$ and $x_\e^-$ denote the center of the interval $I^+_\e$ and $I_\e^-$, respectively.  

We now construct functions $v_\e$ defined on the whole $\mathbb R$, first  in $(-3\tau, 3\tau)$ using a cut-off argument between $u_\e$ and $-1$ and between in $u_\e$ and $1$ and then extended as the constant $-1$ and $1$ in the two half lines $(-\infty, -3\tau]$ and $[3\tau,+\infty)$, respectively.

Let $\phi\colon\mathbb R\to[0,1]$ be such that $\phi\in C^k(\mathbb R)$, $\phi\equiv 1$ in $(-\infty,-1]$ e $\phi\equiv 0$ in $[1,+\infty)$; we set  
\[\phi^-_\e(x)=\phi\Big(\displaystyle\frac{x_\e^{-}-x}{\e^s}\Big) \ \ \hbox{\rm and } \ \  \phi^+_\e(x)=\phi\Big(\displaystyle\frac{x-x_\e^{+}}{\e^s}\Big), 
\]
and define 
$$v_\e(x)=\begin{cases} 
-1 & \hbox{\rm if } x\leq -3\tau\\
(u_\e(x)+1)\phi_\e^-(x)-1 & \hbox{\rm if } -3\tau\leq x\leq 0\\
(u_\e(x)-1)\phi_\e^+(x)+1 & \hbox{\rm if } 0\leq x\leq 3\tau\\
1 & \hbox{\rm if } x\geq 3\tau. 
\end{cases}$$  
Note that $v_\e(x)= -1$ for $x\le x_\e^--\e^s$, $v_\e(x)=u_\e(x)$ for $x_\e^-+\e^s\le x\le x_\e^+-\e^s$, and $v_\e(x)= 1$ for $x\ge x_\e^++\e^s$.

The functions $w_\e(x)=v_\e(\e x)$ are test functions for $\widetilde m_{k+s}$ with $T=3\tau/\e$. We will 
compare their related energies with $F_\e(u_\e,(-3\tau, 3\tau))$. 
\smallskip

We subdivide the rest of the proof in two computations. First, we estimate the energy of $v_\e$ inside $(-3\tau, 3\tau)$, and then, in the second step, we compute the energy of $w_\e$ on the whole line to estimate $\widetilde m_{k+s}$. 

\medskip
\noindent{\em Step $1$.} 
First, we prove that for each fixed $\delta>0$ we have 
\begin{equation}\label{step1}
F_\e(v_\e;(-3\tau,3\tau))\leq (1+\delta)F_\e(u_\e; (-3\tau, 3\tau))+\frac{C}{N}+o(1)_{\e\to 0},   
\end{equation}
where $N\in\mathbb N$ and $C>0$ is independent of $\e$ and $N$. 

We start by proving that 
\begin{eqnarray}\label{02T} 
\e^{2(k+s)-1}[v_\e]^2_{k+s}(-3\tau,3\tau)&\leq&(1+\delta)\e^{2(k+s)-1}[u_\e]^2_{k+s}(-3\tau,3\tau)\nonumber \\
&&+\frac{C}{N}+o(1)_{\e\to 0}. 
\end{eqnarray}
To this end, we preliminary note that, in the case $k\geq 1$, if $I$ is an interval such that $|u_\e-1|\leq 2\eta$ or  $|u_\e+1|\leq 2\eta$ in $I$, then the following estimate holds 
\begin{equation}\label{interp-3}
\|u_\e^{(\ell)}\|^2_{L^2(I)}\leq C\e^{1-2\ell} F_{\e}(u_\e;I) 
\end{equation}
for all $\e$ small enough, with $C$ only depending only $k$ and $s$, for all $\ell\in \{1,\dots, k\}$. To prove this, we can assume that $|u_\e-1|\leq 2\eta$, the other case being completely analogous. In this case, 
\begin{equation}\label{l2}
\|u_\e-1\|^2_{L^2(I)}\leq \frac{1}{\alpha_W}\int_I W(u_\e)\, dx.  
\end{equation} 
Then, interpolation inequality \eqref{interpolazione-eq} and estimate \eqref{stconv} 
allow us to deduce that 
\begin{eqnarray*}
&&\hskip-1.5cm\|u_\e^{(\ell)}\|^2_{L^2(I)}\leq R^2_{k,s}\Big((|I|^{-2\ell}+\e^{-2\ell})\|u_\e-1\|^2_{L^2(I)}+\e^{2(k+s-\ell)}[u_\e]^2_{k+s}(I)\Big) \\
&&\ \ \leq\ 2R^2_{k,s}\max\{1,\alpha_W^{-1}\} \e^{1-2\ell} F_{\e}(u_\e;I)
\end{eqnarray*} 
for all $\ell\in \{1,\dots, k\}$ and $\e\leq |I|$, concluding the proof of \eqref{interp-3}.     

We now can prove \eqref{02T}. For notational convenience, we set
$$\phi_\e(z)=\begin{cases}
\phi_\e^{-}(z) & \hbox{\rm if } z\leq 0\\
\phi_\e^{+}(z) & \hbox{\rm if } z\geq 0.
\end{cases}$$ 
Using a discrete general Leibniz rule and that $|\phi_\e|\leq 1$, we have
\begin{eqnarray}\label{v+}
(v_\e^{(k)}(x)-v_\e^{(k)}(y))^2&\leq& (1+\delta)(u_\e^{(k)}(x)-u_\e^{(k)}(y))^2
\nonumber\\
&&+c_k\Big(1+\frac{1}{\delta}\Big)\sum_{\ell=0}^{k-1}(u_\e^{(\ell)}(x)-u_\e^{(\ell)}(y))^2
(\phi_\e^{(k-\ell)}(x))^2\nonumber\\
&&+c_k\Big(1+\frac{1}{\delta}\Big)\sum_{\ell=1}^{k}(u_\e^{(\ell)}(y))^2
(\phi_\e^{(k-\ell)}(x)-\phi_\e^{(k-\ell)}(y))^2\nonumber\\
&&+c_k\Big(1+\frac{1}{\delta}\Big)(u_\e(y)-1)^2
(\phi_\e^{(k)}(x)-\phi_\e^{(k)}(y))^2,
\end{eqnarray}
where $c_k$ is a positive constant depending only on $k$.

We now estimate the $s$-seminorm of $v_\e^{(k)}$ on $(-3\tau, 3\tau)$ using inequality \eqref{v+}.
We subdivide the computation into integrals in the following three sets.
\begin{enumerate}
\item both $x$ e $y$ not belonging to $I_\e^{+}\cup I_\e^{-}$;
\item $y\in I^+_\e$ and $x\in (-3\tau,3\tau)\setminus I^+_\e$. This estimate also holds for $y\in I^-_\e$ and $x\in (-3\tau,3\tau)\setminus I^-_\e$;
\item $x,y\in I_\e^+$. This estimate also holds for $x,y\in I_\e^-$.
\end{enumerate}

In the following computations, $C$ will denote a positive constant not depending on $N$ and  $\e$, possibly changing from line to line.

\medskip

\noindent{\em Case $1\colon$} $x,y\in A_\e:=(-3\tau,3\tau)\setminus (I_\e^+\cup I_\e^-)$. 

\smallskip
We preliminarily note that, if both $x$ and $y$ do not belong to $I_\e^+\cup I_\e^-$, then 
$\phi_\e^{(k-\ell)}(x)=\phi_\e^{(k-\ell)}(y)=0$ for all $\ell\in\{0,\dots,k-1\}$, and \eqref{v+} becomes 
\begin{eqnarray}\label{v+outs}
(v_\e^{(k)}(x)-v_\e^{(k)}(y))^2&\leq&(1+\delta)(u_\e^{(k)}(x)-u_\e^{(k)}(y))^2
\nonumber\\
&&+c_k\Big(1+\frac{1}{\delta}\Big)(u_\e^{(k)}(y))^2
(\phi_\e(x)-\phi_\e(y))^2.
\end{eqnarray}
Setting $I_\e^+=(a^+_\e,b^+_\e)$ and $I_\e^-=(b^-_\e,a^-_\e)$, it is sufficient to consider the case $x\in (a^-_\e, a^+_\e)$, $y\in (-3\tau, b_\e^-)\cup (b_\e^+,3\tau)$, where $|\phi_\e(x)-\phi_\e(y)|=1$. 
Indeed, if $y\in (a^-_\e, a^+_\e), x\in (-3\tau, b_\e^-)\cup (b_\e^+,3\tau)$, we can reason in the same way by exchanging the role of $x$ and $y$ in \eqref{v+outs}, while 
if both $x$ and $y$ belong to $(a^-_\e, a^+_\e)$ or to $(-3\tau, b_\e^-)\cup (b_\e^+,3\tau)$ then $\phi_\e(x)=\phi_\e(y)$.

We first estimate
\begin{eqnarray}
\int_{a^-_\e}^{a^+_\e}\int_{b^+_\e}^{3\tau}\frac{|u_\e^{(k)}(y)|^2}{|x-y|^{1+2s}}\, dx\,dy
&\leq& 
\frac{a_\e^+-a_\e^-}{(b^+_\e-a^+_\e)^{1+2s}}\|u_\e^{(k)}\|^2_{L^2(\beta_\e,3\tau)}\nonumber\\
&\leq& \frac{4N^{1+2s}}{\tau^{2s}}\|u_\e^{(k)}\|^2_{L^2(\tau,3\tau)}\nonumber\\
&\leq&CN^{1+2s}\e^{1-2k}\nonumber, 
\end{eqnarray} 
where in the last inequality we have applied \eqref{interp-3} with $\ell=k$ and $I=(\tau,3\tau)$,  since in this interval  $|u_\e-1|\leq 2\eta$. 

Analogously, applying \eqref{interp-3} in the interval $(-3\tau,-\tau)$, where $|u_\e+1|\leq 2\eta$,  we have
\begin{eqnarray}
\int_{a^-_\e}^{a^+_\e}\int_{-3\tau}^{b^-_\e}\frac{|u_\e^{(k)}(y)|^2}{|x-y|^{1+2s}}\, dx\,dy
\ \leq\ CN^{1+2s}\e^{1-2k}\nonumber. 
\end{eqnarray} 
We then get
\begin{eqnarray}\label{dd}
&&\hspace{-1cm}\int_{
A_\e}\int_{
A_\e}\frac{|v_\e^{(k)}(x)-v_\e^{(k)}(y)|^2}{|x-y|^{1+2s}}\, dx\, dy\nonumber \\
&&\leq(1+\delta)\int_{
A_\e}\int_{
A_\e}\frac{|u_\e^{(k)}(x)-u_\e^{(k)}(y)|^2}{|x-y|^{1+2s}}\, dx\, dy
+CN^{1+2s}\e^{1-2k}. 
\end{eqnarray}

\smallskip
\noindent{\em Case $2\colon$ $x\in (-3\tau,3\tau)\setminus I_\e^+$ and $y\in I_\e^+$.}  

\smallskip
We first note  that, if $x,y\geq 0$,  
\begin{eqnarray*}
\frac{|\phi_\e^{(k-\ell)}(x)-\phi_\e^{(k-\ell)}(y)|^2}{|x-y|^2}&=&\e^{-2s(k-\ell)}\frac{\big|\phi^{(k-\ell)}\big(\frac{x-x_\e^+}{\e^s}\big)-\phi^{(k-\ell)}\big(\frac{y-x_\e^+}{\e^s}\big)\big|^2}{\e^{2s}\big|\frac{x-x_\e^+}{\e^s}-\frac{y-x_\e^+}{\e^s}\big|^2}\\
&\leq&\|\phi\|_{C^k(\mathbb R)}^2\, \e^{-2s(k-\ell+1)}
\end{eqnarray*}
for all $\ell\in\{0,\dots,k\}$. Then, for all $y\geq 0$ 
the following estimate holds: 
\begin{eqnarray}\label{phi+}
\int_{0}^{3\tau}\frac{|\phi_\e^{(k-\ell)}(x)-\phi_\e^{(k-\ell)}(y)|^2}{|x-y|^{1+2s}}\, dx&\leq& \frac{\|\phi\|_k^2}{\e^{2s(k-\ell+1)}}\int_0^{3\tau}|x-y|^{1-2s}\, dx\nonumber\\
&\leq& C\e^{-2s(k-\ell+1)}.
\end{eqnarray} 

If $x\leq 0$ and $y\geq 0$, it is sufficient to observe that 
\begin{equation*}
|\phi_\e^{(k-\ell)}(x)- \phi_\e^{(k-\ell)}(y)|^2\leq 2\|\phi\|^2_{C^k(\mathbb R)}\e^{-2s(k-\ell)} 
\end{equation*} 
for all $\ell\in\{0,\dots, k\}$. 
Since $|x-y|\geq \tau$ if $x\in(-3\tau,0)$ and $y\in I_\e^+$, we then have \begin{equation}\label{phi-}
\int_{-3\tau}^{0}\frac{|\phi_\e^{(k-\ell)}(x)- \phi_\e^{(k-\ell)}(y)|^2}{|x-y|^{1+2s}}\, dx
\leq C\e^{-2s(k-\ell)} 
\end{equation} 
for all $y\in I_\e^+$ and $\ell\in\{0,\dots, k\}$. 

Inequalities  
\eqref{phi+} and \eqref{phi-} allow us to estimate 
\begin{eqnarray*}
&&\hspace{-2cm}\int_{I_\e^+}\int_{(-3\tau,3\tau)\setminus I_\e^+} |u_\e^{(\ell)}(y)|^2\frac{|\phi_\e^{(k-\ell)}(x)-\phi_\e^{(k-\ell)}(y)|^2}{|x-y|^{1+2s}}\, dx\, dy
\nonumber\\
&&\leq C(\e^{-2s(k-\ell)} 
+\e^{-2s(k-\ell+1)})\|u_\e^{(\ell)}\|^2_{L^2(I_\e^+)} 
\end{eqnarray*}
for all $\ell\in\{1,\dots, k\}$, and 
\begin{eqnarray}\label{g0}
&&\hspace{-2cm}\int_{I_\e^+}\int_{(-3\tau,3\tau)\setminus I_\e^+} |u_\e(y)-1|^2\frac{|\phi_\e^{(k)}(x)-\phi_\e^{(k)}(y)|^2}{|x-y|^{1+2s}}\, dx\, dy
\nonumber\\
&&\leq C(\e^{-2sk} 
+\e^{-2s(k+1)})\|u_\e-1\|^2_{L^2(I_\e^+)}.
\end{eqnarray} 
Since $|u_\e-1|\leq 2\eta$ in $I_\e^+$, 
we can apply the bound for $\|u_\e^{(\ell)}\|^2_{L^2(I)}$ given in \eqref{interp-3} with $I=I_\e^+$, obtaining 
\begin{eqnarray}\label{rf}
&&\hspace{-2cm}\int_{I_\e^+}\int_{(-3\tau,3\tau)\setminus I_\e^+} |u_\e^{(\ell)}(y)|^2\frac{|\phi_\e^{(k-\ell)}(x)-\phi_\e^{(k-\ell)}(y)|^2}{|x-y|^{1+2s}}\, dx\, dy
\nonumber\\
&&\leq\ C\e^{-2s(k-\ell+1)+1-2\ell}F_\e(u_\e;I_\e^+)\ 
\leq\ \frac{C}{N}\e^{-2k-2s+1}
\end{eqnarray}
for all $\ell\in\{1,\dots, k\}$, since $F_\e(u_\e;I_\e^+)\leq \frac{S}{N}$. 
As for \eqref{g0}, we can apply \eqref{l2}  to obtain 
\begin{eqnarray}\label{g0f}
&&\hspace{-2cm}\int_{I_\e^+}\int_{(-3\tau,3\tau)\setminus I_\e^+} |u_\e(y)-1|^2\frac{|\phi_\e^{(k)}(x)-\phi_\e^{(k)}(y)|^2}{|x-y|^{1+2s}}\, dx\, dy
\nonumber\\
&&\leq\ C
\e^{-2sk-2s+1}F_\e(u_\e;I_\e^+)\ \leq\ \frac{C}{N} 
\e^{-2sk-2s+1}.
\end{eqnarray} 

We now note that, for all $\ell\in\{0,\dots, k-1\}$, $|\phi_\e^{(k-\ell)}(x)|\leq \e^{-s(k-\ell)} \|\phi\|_{C^k(\mathbb R)}$, and that in particular it vanishes if $x\in (0,3\tau)\setminus I_\e^+$ and $x\in (-3\tau,0)\setminus I_\e^-$.   
Since $|x-y|\geq \tau$ if $x\in (-3\tau,0)$ and $y\in I_\e^+$, we then have 
\begin{eqnarray}\label{f}
&&\hspace{-2cm}\int_{I_\e^+}\int_{(-3\tau,3\tau)\setminus I_\e^+} 
\frac{|u_\e^{(\ell)}(x)-u_\e^{(\ell)}(y)|^2}{|x-y|^{1+2s}} |\phi_\e^{(k-\ell)}(x)|^2\, dx\, dy
\nonumber\\
&&\leq C\e^{-2s(k-\ell)} 
(\|u_\e^{(\ell)}\|^2_{L^2(I_\e^-)}+\|u_\e^{(\ell)}\|^2_{L^2(I_\e^+)})
\ \leq \ \frac{C}{N}\e^{1-2k} 
\end{eqnarray}
for all $\ell\in\{1,\dots, k-1\}$, where we again used estimate \eqref{interp-3} and the bound on $F_\e(u_\e;I_\e^+)$ and on $F_\e(u_\e;I_\e^-)$. 
As for $\ell=0$, let $X_\e$ denote the subset of $I_\e^-$ where $\phi_\e^{(k)}\neq 0$, and note that its measure is not greater than $\frac{\tau\e^s}{N}$. Then,  
\begin{eqnarray}\label{f0}
&&\hspace{-2cm}\int_{I_\e^+}\int_{(-3\tau,3\tau)\setminus I_\e^+} 
\frac{|u_\e(x)-u_\e(y)|^2}{|x-y|^{1+2s}} |\phi_\e^{(k)}(x)|^2\, dx\, dy
\nonumber\\
&&\leq C\Big(\int_{I_\e^+} \int_{X_\e} |u_\e(x)|^2\, dx\, dy+
\int_{X_\e}\int_{I_\e^+}  |u_\e(y)|^2\, dy\, dx\Big)
\nonumber\\
&&\leq C\e^{-2sk} |X_\e| |I_\e^+|
\ \leq \ \frac{C}{N}\e^{1-2k},  
\end{eqnarray}
since $u_\e$ is equibounded in $I_\e^+\cup I_\e^-$. 

Recalling \eqref{v+}, by \eqref{rf}, \eqref{g0f}, \eqref{f} and \eqref{f0}, we deduce that 
\begin{eqnarray}\label{tau+}
&&\hspace{-12mm}\int_{I_\e^+}\int_{(-3\tau,3\tau)\setminus I_\e^+}\frac{|v_\e^{(k)}(x)-v_\e^{(k)}(y)|^2}{|x-y|^{1+2s}}\, dx\, dy \nonumber\\
&\hspace{-2mm}\leq&\hspace{-2mm}(1+\delta) \int_{I_\e^+}\int_{(-3\tau,3\tau)\setminus I_\e^+}\frac{|u_\e^{(k)}(x)-u_\e^{(k)}(y)|^2}{|x-y|^{1+2s}}\, dx\, dy+\frac{C}{N}\e^{-2k-2s+1}. 
\end{eqnarray} 

\smallskip 

\noindent{\em Case $3\colon$ $x,y\in I_\e^+$.}

\smallskip
 In this case, we also have to consider the sum up to $k-1$ of the non-vanishing terms
$(u_\e^{(\ell)}(x)-u_\e^{(\ell)}(y))^2((\phi_\e^+)^{(k-\ell)}(x))^2$. 
Noting that 
\begin{eqnarray*}
&&\hspace{-1cm}\int_{I_\e^+}\int_{I_\e^+} \frac{|u_\e^{(\ell)}(x)-u_\e^{(\ell)}(y)|^2}{|x-y|^{1+2s}} ((\phi_\e^+)^{(k-\ell)}(x))^2\, dx\, dy\\
&&\leq \e^{-2s(k-\ell)}\|\phi\|_{C^k(\mathbb R)}^2[u_\e]^2_{\ell+s}(I_\e^+),  
\end{eqnarray*}
we can apply \cite[Theorem 1.25]{leofrac} to estimate the $s$-seminorm of $u^{(\ell)}_\e$ as follows: 
\begin{equation}\label{semi}
[u^{(\ell)}_\e]^2_s(I_\e^+)\leq  \hat C(\|u_\e^{(\ell)}\|^2_{L^2(I_\e^+)}+\|u_\e^{(\ell+1)}\|^2_{L^2(I_\e^+)})
\end{equation}
for all $\ell\in\{1,\dots, k-1\}$,  
where the positive constant $\hat C$ only depends on $\ell,s$ and on the length of the interval $I_\e^+$ (which is independent of $\e$).  Note that for $\ell=0$ we have that 
\begin{equation}\label{semi0}
[u_\e]^2_{s}(I_\e^+)\leq  \hat C(\|u_\e-1\|^2_{L^2(I_\e^+)}+\|u_\e^{\prime}\|^2_{L^2(I_\e^+)}). 
\end{equation}
Since \eqref{interp-3} and \eqref{l2} hold, we then get 
\begin{equation}\label{semibis}
[u_\e]^2_{\ell+s}(I_\e^+)\leq  C\e^{-1-2\ell} F_{\e}(u_\e;I_\e^+)\leq \frac{C}{N}\e^{-1-2\ell}, 
\end{equation}
where $C$ does not depend on $\e$. 
We then conclude that 
\begin{eqnarray}\label{dop}
\int_{I_\e^+}\int_{I_\e^+} \frac{|u_\e^{(\ell)}(x)-u_\e^{(\ell)}(y)|^2}{|x-y|^{1+2s}} ((\phi_\e^+)^{(k-\ell)}(x))^2\, dx\, dy&\leq&\frac{C}{N}\e^{-2s(k-\ell)-1-2\ell}\nonumber\\
&\leq&\frac{C}{N}\e^{-2k-2s+1}.  
\end{eqnarray}

Recalling \eqref{phi+}, and applying the bound for $\|u_\e^{(\ell)}\|^2_{L^2(I)}$ given in \eqref{interp-3} with $I=I_\e^+$, similarly to the previous case we deduce  
\begin{eqnarray}\label{rfi}
&&\hspace{-2cm}\int_{I_\e^+}\int_{I_\e^+} |u_\e^{(\ell)}(y)|^2\frac{|\phi_\e^{(k-\ell)}(x)-\phi_\e^{(k-\ell)}(y)|^2}{|x-y|^{1+2s}}\, dx\, dy
\nonumber\\
&&\leq\ C\e^{-2s(k-\ell+1)+1-2\ell}F_\e(u_\e;I_\e^+)\ 
\leq\ \frac{C}{N}\e^{-2k-2s+1}
\end{eqnarray}
for all $\ell\in\{1,\dots, k\}$, since $F_\e(u_\e;I_\e^+)\leq \frac{S}{N}$. 
Moreover, by using \eqref{l2} we get  
\begin{eqnarray}\label{g0fi}
&&\hspace{-2cm}\int_{I_\e^+}\int_{I_\e^+} |u_\e(y)-1|^2\frac{|\phi_\e^{(k)}(x)-\phi_\e^{(k)}(y)|^2}{|x-y|^{1+2s}}\, dx\, dy
\nonumber\\
&&\leq C
\e^{-2sk-2s+1}F_\e(u_\e;I_\e^+)\ \leq \frac{C}{N} 
\e^{-2sk-2s+1}.
\end{eqnarray}

By collecting \eqref{dop}, \eqref{rfi} and \eqref{g0fi}, it follows that 
\begin{eqnarray}\label{pp}
&&\hspace{-1cm}\int_{I_\e^+}\int_{I_\e^+}\frac{|v^{(k)}_\e(x)-v^{(k)}_\e(y)|}{|x-y|^{1+2s}}\, dx\, dy\nonumber\\
&\leq&(1+\delta)\int_{I_\e^+}\int_{I_\e^+}\frac{|u^{(k)}_\e(x)-u^{(k)}_\e(y)|}{|x-y|^{1+2s}}\, dx\, dy+\frac{C}{N}\e^{-2k-2s+1}. 
\end{eqnarray}

\medskip

We now 
estimate the part of the energy of $v_\e$ involving the double-well potential $W$. Since $v_\e=u_\e$ in $(-3\tau,3\tau)\setminus (I_\e^+\cup I_\e^-)$, and recalling the bound on the energy \eqref{enb}, it is sufficient to prove that  
there exists $C>0$ independent of $\e$ such that 
\begin{equation}\label{sW}
\frac{1}{\e}\int_{I_\e^+\cup I_\e^-}W(v_\e)\, dx\leq  \frac{C}{\e}\int_{I_\e^+\cup I_\e^-}W(u_\e)\, dx. 
\end{equation} 
The local hypotheses on $W$ in a neighbourhood of the minimum points ensure that 
$$\alpha_W (z-1)^2\leq W(z)\leq \gamma_W(z-1)^2$$  
for $|z-1|<\overline\eta$. 
Since for $x\in I_\e^+$ we have 
$$(u_\e(x)-1)\phi_\e^+(x)+1\in [\min\{u_\e(x),1\}, \max\{u_\e(x),1\}]\subseteq[1-2\eta,1+2\eta],$$ 
by choosing $\eta$ such that $2\eta<\overline\eta$ we get 
\begin{eqnarray*}
W(v_\e(x))&=&W((u_\e(x)-1)\phi_\e^+(x)+1)
\ \leq\ \gamma_W((u_\e(x)-1)\phi_\e^+(x))^2\\
&\leq& \gamma_W(u_\e(x)-1)^2 
\ \leq \ \frac{\gamma_W}{\alpha_W}W(u_\e(x)), 
\end{eqnarray*}
and, recalling \eqref{enb}, 
\begin{eqnarray*}
\frac{1}{\e}\int_{I_\e^+}W(v_\e)\, dx\leq\frac{\gamma_W}{\alpha_W}\frac{1}{\e}\int_{I_\e^+}W(u_\e(x))\, dx\leq \frac{C}{N}. 
\end{eqnarray*}
A completely analogous argument can be applied in $I_\e^-$.  

Since \eqref{02T} holds, this concludes the proof of estimate \eqref{step1}. 

\bigskip
\noindent{\em Step $2$.} 
We now can finally prove the desired estimate. More precisely, we now define $w_\e(x)=v_\e(\e x)$, and prove that 
\begin{eqnarray}\nonumber\label{dis-co-1}
&&\hspace{-1cm}\int_{\mathbb R} W(w_\e)\, dx+\int_{\mathbb R}\!\int_{\mathbb R} \frac{|w_\e^{(k)}(x)-w_\e^{(k)}(y)|^2}{|x-y|^{1+2s}}\, dx\, dy\\
&&\leq 
F_\e(v_\e;(-3\tau,3\tau))+c_\sigma+o(1)_{\e\to 0}, 
\end{eqnarray} 
with $c_\sigma$ independent of $\e$  and such that $c_\sigma\to 0$ as $\sigma\to 0$. 
The function $w_\e$ is an admissible test function for the problem defining $\widetilde m_{k+s}$ since by construction $v_\e(x)=1$ if $x\geq 2\tau$ and $v_\e(x)=-1$ if $x\leq -2\tau$. Hence, \eqref{dis-co-1} ensures we have   
\begin{eqnarray}\label{dis-co-2}
\widetilde m_{k+s} \leq F_\e(v_\e;(-3\tau,3\tau))+c_\sigma+o(1)_{\e\to 0}.  
\end{eqnarray}

To prove \eqref{dis-co-1}, we note that 
\begin{eqnarray*}
\int_{-\infty}^{+\infty} W(w_\e)\, dx=\int_{-\frac{3\tau}{\e}}^{\frac{3\tau}{\e}} W(w_\e)\, dx=\frac{1}{\e}\int_{-3\tau}^{3\tau} W(v_\e)\, dx, 
\end{eqnarray*}
and that 
\begin{eqnarray*}
&&\hskip-2cm\int_{-\frac{3\tau}{\e}}^{\frac{3\tau}{\e}} \int_{-\frac{3\tau}{\e}}^{\frac{3\tau}{\e}} \frac{|w_\e^{(k)}(x)-w_\e^{(k)}(y)|^2}{|x-y|^{1+2s}}\, dx\, dy\\
&&=\e^{2(k+s)-1}\int_{-3\tau}^{3\tau} \int_{-3\tau}^{3\tau} \frac{|v_\e^{(k)}(x)-v_\e^{(k)}(y)|^2}{|x-y|^{1+2s}}\, dx\, dy.  
\end{eqnarray*}
Hence, we only have to estimate the integral of $\frac{|w_\e^{(k)}(x)-w_\e^{(k)}(y)|^2}{|x-y|^{1+2s}}$ in the sets $(\frac{3\tau}{\e},+\infty)\times (-\frac{3\tau}{\e},\frac{3\tau}{\e})$, $(-\infty, -\frac{3\tau}{\e})\times (-\frac{3\tau}{\e},\frac{3\tau}{\e})$ and $(\frac{3\tau}{\e},+\infty)\times (-\infty, -\frac{3\tau}{\e})$. 

\smallskip
We first treat the case $k\geq 1$, for which $w_\e^{(k)}(x)=0$ if $|x|\ge \frac{2\tau}{\e}$, and write 
\begin{eqnarray}\label{sti-ca-1}\nonumber
&&\hskip-2cm\int_{\frac{3\tau}{\e}}^{+\infty}\!\!\int_{-\frac{3\tau}{\e}}^{\frac{3\tau}{\e}} \frac{|w_\e^{(k)}(x)-w_\e^{(k)}(y)|^2}{|x-y|^{1+2s}}\, dx\, dy\\&=&\nonumber\int_{-\frac{2\tau}{\e}}^{\frac{2\tau}{\e}} |w_\e^{(k)}(x)|^2 \!\int_{\frac{3\tau}{\e}}^{+\infty}\frac{1}{(y-x)^{1+2s}}\, dy\, dx\\
&\leq&\frac{1}{2s} \frac{\e^{2s}}{\tau^{2s}}\int_{-\frac{2\tau}{\e}}^{\frac{2\tau}{\e}} |w_\e^{(k)}(x)|^2 
\, dx.
\end{eqnarray}

We now recall that by  \eqref{dtj} there exist $\sigma^-\in[-2\sigma,-\sigma]$ and $
\sigma^+\in[\sigma,2\sigma]$ such that $|u_\e^{(k-1)}(\sigma^{\pm})|=|v_\e^{(k-1)}(\sigma^{\pm})|\leq \e^{1-k}$. 
Since $||v_\e|-1|\leq 2\eta$ in $A:=(-3\tau,3\tau)\setminus (\sigma^-,\sigma^+)$, we may use  interpolation inequality \eqref{interp-3} in $(-3\tau,\sigma^-)$ and in $(\sigma^+, 3\tau)$
 to estimate 
\begin{equation*}
\|v_\e^{(k)}\|^2_{L^2(-3\tau,\sigma^-)}+\|v_\e^{(k)}\|^2_{L^2(\sigma^+, 3\tau)}\leq C\e^{1-2k} \big(F_{\e}(v_\e;(-3\tau,\sigma^-))+F_{\e}(v_\e;(\sigma^+, 3\tau))\big).
\end{equation*}
Since $w_\e^{(k)}(x)=\e^{k}v_\e^{(k)}(\e x)$, we get that $\int_{\frac{A}{\e}} |w_\e^{(k)}(x)|^2\, dx$ is equibounded. 
Recalling that 
$|v_\e^{(k-1)}(\sigma^{\pm})|\leq \e^{1-k}$, we  have that
$$|w_\e^{(k-1)}(\sigma^+)-w_\e^{k-1}(\sigma^-)|\leq 2$$ and that $[w_\e]^2_{k+s}$ is equibounded by $S$, using \eqref{stimaseminormaeq} with $J=I=(\frac{\sigma^-}{\e},\frac{\sigma^+}{\e})$ and $u=w_\e^{(k)}$, we deduce that 
\begin{eqnarray*}
\|w_\e^{(k)}\|^2_{L^2(\frac{\sigma^-}{\e},\frac{\sigma^+}{\e})}\leq C\Big(\frac{\e}{\sigma} + \frac{\sigma^{2s}}{\e^{2s}}\Big), 
\end{eqnarray*}
since $\sigma^-\in[-2\sigma,-\sigma]$ and $
\sigma^+\in[\sigma,2\sigma]$. 
Hence, using \eqref{sti-ca-1} and the fact that $\int_{\frac{A}{\e}} |w_\e^{(k)}(x)|^2\, dx$ is equibounded, we obtain 
\begin{eqnarray*}
\int_{\frac{3\tau}{\e}}^{+\infty}\int_{-\frac{3\tau}{\e}}^{\frac{3\tau}{\e}} \frac{|w_\e^{(k)}(x)-w_\e^{(k)}(y)|^2}{|x-y|^{1+2s}}\, dx\, dy
&\leq&C \frac{\e^{2s}}{\tau^{2s}} \Big(1+\frac{\e}{\sigma}+\frac{\sigma^{2s}}{\e^{2s}}\Big)
\\
&=&C\sigma^{2s}+o(1)_{\e\to 0}.
\end{eqnarray*} 
The analogous argument gives the same bound for the integral on $(\frac{3\tau}{\e},+\infty)\times (-\frac{3\tau}{\e},\frac{3\tau}{\e})$. 
As for the integral in $(\frac{3\tau}{\e},+\infty)\times (-\infty, -\frac{3\tau}{\e})$, we have that 
\begin{eqnarray*}
\int_{-\infty}^{-\frac{3\tau}{\e}} \int_{\frac{3\tau}{\e}}^{+\infty}  \frac{|w_\e^{(k)}(x)-w_\e^{(k)}(y)|^2}{|x-y|^{1+2s}}\, dx\, dy=0, 
\end{eqnarray*}
concluding the proof of \eqref{dis-co-1} if $k\ge 1$.

\smallskip
We now treat the case $k=0$, for which we have to take into account that $w_\e(x)=1$ if $x\ge \frac{2\tau}{\e}$ and $w_\e(x)=-1$ if $x\le -\frac{2\tau}{\e}$.
We then have
\begin{eqnarray*}
&&\hskip-2cm\int_{\frac{3\tau}{\e}}^{+\infty}\!\!\int_{-\frac{3\tau}{\e}}^{\frac{3\tau}{\e}} \frac{|w_\e(x)-w_\e(y)|^2}{|x-y|^{1+2s}}\, dx\, dy\\&=&\int_{-\frac{2\tau}{\e}}^{\frac{2\tau}{\e}} |w_\e(x)-1|^2 \!\int_{\frac{3\tau}{\e}}^{+\infty}\frac{1}{(y-x)^{1+2s}}\, dy\, dx\\
&\leq&\frac{1}{2s} \frac{\e^{2s}}{\tau^{2s}}\int_{-\frac{2\tau}{\e}}^{\frac{2\tau}{\e}} |w_\e(x)-1|^2 
\, dx.
\end{eqnarray*}

Let $x\in (-\sigma, \sigma)$, and consider a neighbourhood of $x$ with length strictly greater than 
$\frac{2S\e }{\alpha_W\eta^2}$ and included in $(-\tau,\tau)$. 
Since \eqref{mis} holds; that is, the measure of the set where
$||u_\e(t)|-1|\geq \eta$ is less than $\frac{2S\e}{\alpha_W\eta^2}$, then, noting that 
$u_\e=v_\e$ in $(-\tau,\tau)$, there exists a point $y(x)$ such that $|x-y(x)|<\frac{2S\e }{\alpha_W\eta^2}$ and $|v_\e(y(x))|<1+\eta$. We now use the embedding of $H^s$ into $C^{0,s-\frac12}$ to deduce that 
\begin{equation}
|v_\e(x)-v_\e(y)|^2\leq [v_\e]^2_s(-\tau,\tau) |x-y(x)|^{2s-1}\leq \e^{1-2s}S \Big(\frac{2S\e }{\alpha_W\eta^2}\Big)^{2s-1}. 
\end{equation}
Then, $v_\e$ is equibounded in $L^\infty(-\sigma,\sigma)$. Since it is equibounded by $1+2\eta$ in the remaining part of $(-3\tau,3\tau)$, it follows that $\|w_\e\|_{L^{\infty}(\mathbb R)}$ is equibounded. Note that alternatively, we can also argue by truncation, since the seminorm $[u]_s$ decreases by truncation.
Hence, we obtain 
\begin{eqnarray*}
\int_{\frac{3\tau}{\e}}^{+\infty}\!\!\int_{-\frac{3\tau}{\e}}^{\frac{3\tau}{\e}} \frac{|w_\e(x)-w_\e(y)|^2}{|x-y|^{1+2s}}\, dx\, dy
&\leq&\frac{1}{2s} \frac{\e^{2s}}{\tau^{2s}}\frac{4\tau}{\e} \|w_\e(x)-1\|_{L^\infty(\mathbb R)}^2 \\
&\leq& C\e^{2s-1}, 
\end{eqnarray*}
which is infinitesimal as $\e\to 0$ since $2s>1$. 
The same argument gives the same bound for the integral in $(\frac{3\tau}{\e},+\infty)\times (-\frac{3\tau}{\e},\frac{3\tau}{\e})$. 

As for the integral in $(\frac{3\tau}{\e},+\infty)\times (-\infty, -\frac{3\tau}{\e})$, recalling that $2s>1$ we have that 
\begin{eqnarray*}
\int_{-\infty}^{-\frac{3\tau}{\e}} \int_{\frac{3\tau}{\e}}^{+\infty}  \frac{|w_\e(x)-w_\e(y)|^2}{|x-y|^{1+2s}}\, dx\, dy&=&4\int_{-\infty}^{-\frac{3\tau}{\e}} \int_{\frac{3\tau}{\e}}^{+\infty}  \frac{1}{|x-y|^{1+2s}}\, dx\, dy\\
&=&\frac{2}{s(2s-1)}\Big(\frac{6\tau}{\e}\Big)^{1-2s}, 
\end{eqnarray*}
which is infinitesimal as $\e\to 0$ since $2s>1$, 
concluding the proof of \eqref{dis-co-1} if $k=0$.

\bigskip
\noindent{\em Conclusion of the proof.} Noting that the argument can be repeated for each jump point $x_i$ of $u$, using \eqref{pp} 
and \eqref{dis-co-2} we obtain that
\begin{eqnarray*}
\liminf_{\e\to 0} F_\e(u_\e)&\ge &\liminf_{\e\to 0} \sum_{i=1}^M F_\e(u_\e;(x_i-3\tau, x_i+3\tau))
\\ & =&\frac1{1+\delta}\Big( \liminf_{\e\to 0} \sum_{i=1}^M F_\e(v_\e;(x_i-3\tau, x_i+3\tau))-M \frac CN\Big)\\
&\ge& \frac1{1+\delta}\Big(\sum_{i=1}^M \liminf_{\e\to 0}  F_\e(v_\e;(x_i-3\tau, x_i+3\tau))- M\frac CN\Big)\\
&\ge& \frac1{1+\delta}\Big(M\widetilde m_{k+s} -M\frac CN -Mc_\sigma\Big),
\end{eqnarray*}
which concludes the proof by the arbitrariness of $\delta>0$ and $N\in\mathbb N$ and recalling that $c_\sigma\to 0$ as $\sigma\to 0$.
\end{proof}

\section{Upper bound and optimal-profile formulas}

We now can conclude the proof of Theorem \ref{main}.

\begin{proposition}
If $u\colon(0,1)\to\{-1,1\}$ has a finite discontinuity set $S(u)$, then there exists a family $u_\e$ such that $u_\e\to u$ in $L^1(0,1)$ and such that $F_\e(u_\e)$ converges to $\widetilde m_{k+s}\#S(u)$.
\end{proposition}

\begin{proof}
Up to a diagonal argument, and taking the lower bound into account, it suffices to prove that for each $\sigma>0$ there exists a sequence $u_\e$ such that $u_\e\to u$ in $L^1(0,1)$ and 
\begin{equation}
\limsup_{\e\to 0}F_\e(u_\e)\le\widetilde m_{k+s}\#S(u)+\sigma.
\end{equation}

\smallskip
Let $\eta>0$ be fixed and let $T>0$ and $v\in H^{k+s}_{\rm loc} (\mathbb R)$ be such that 
$$
\int_{-\infty}^{+\infty} W(v)dx+[v^{(k)}]_{s}(\mathbb R)\le \widetilde m_{k+s}+\eta
$$
and $v(t)=1$ for $t\ge T$, $v(t)=-1$ for $t\le -T$.

If  $N=\#S(u)$ and $S(u)=\{t_1,\ldots, t_N\}$ with $t_i<t_{i+1}$, we define the sets
$$
S_+=\{i\in\{1,\ldots,N\}: u(t_i^+)=1\},\quad S_-=\{i\in\{1,\ldots,N\}: u(t_i^-)=1\}.
$$
Note that one of these two sets is just the set of even $i$ and the other the set of odd $i\in\{1,\ldots,N\}$.
For $\e$ small enough we define $u_\e\colon(0,1)\to \mathbb R$ as 
$$
u_\e(t)=\begin{cases}
v(\frac{t-t_i}\e) &\hbox{ if }|t-t_i|\le\e T, i\in S_+\\
v(\frac{t_i-t}\e) &\hbox{ if }|t-t_i|\le\e T, i\in S_-\\
u(t) &\hbox{ otherwise.}
\end{cases}
$$
The functions thus defined belong to $H^{k+s}(0,1)$. We now estimate $F_\e(u_\e)$.

If we set $x_0=0$, $x_i=\frac{t_i+t_{i+1}}2$ for $i\in\{1,\ldots, N-1\}$ and $x_N=1$, 
then we can write
\begin{eqnarray*}
&&F_\e(u_\e)=\sum_{i=1}^N\frac1\e\int_{t_i-\e T}^{t_i+\e T}W(u_\e)dx 
\\&&\hskip1.5cm+\e^{2(k+s)-1}
\sum_{i=1}^N\sum_{j=1}^N\int_{x_{i-1}}^{x_i}\int_{x_{j-1}}^{x_j}\frac{|u^{(k)}_\e(x)-
u^{(k)}_\e(y)|^2}{|x-y|^{1+2s}}dx\,dy.
\end{eqnarray*}
Changing variables, we have 
$$
\frac1\e\int_{t_i-\e T}^{t_i+\e T}W(u_\e)dt= \int_{-T}^T W(v)dt
$$
and 
$$
\e^{2(k+s)-1}\int_{x_{i-1}}^{x_i}\int_{x_{i-1}}^{x_i}\frac{|u^{(k)}_\e(x)-
u^{(k)}_\e(y)|^2}{|x-y|^{1+2s}}dx\,dy\le [v^{(k)}]_{s}(\mathbb R)
$$
 for all $i\in\{1,\ldots, N\}$.
Hence, the claim follows once we prove that each term
\begin{equation}\label{Iij} I^\e_{ij}:=
\e^{2(k+s)-1}\int_{x_{i-1}}^{x_i}\int_{x_{j-1}}^{x_j}\frac{|u^{(k)}_\e(x)-
u^{(k)}_\e(y)|^2}{|x-y|^{1+2s}}dy\,dx
\end{equation}
in the double sum with $i\neq j$ is asymptotically negligible. After a change of variables we can write
$$
I^\e_{ij}=\int_{\frac{x_{i-1}-t_i}\e}^{\frac{x_i-t_i}\e}\int_{\frac{x_{j-1}-t_j}\e}^{\frac{x_j-t_j}\e} \frac{|v^{(k)}(x)-
v^{(k)}(\pm y)|^2}{\big|y-x+\frac{t_j-t_i}\e \big|^{1+2s}}dy\,dx,
$$
with the plus sign if $i-j$ is even, and with the minus sign otherwise. The term $t_i-t_j$ in the denominator shows that the double integrals with the largest values are those with consecutive indices, then it suffices to treat the case when $i=1$ and $j=2$.  It is not restrictive to suppose that $1\in S_+$. We then have to estimate 
$$
I^\e=\int_{\frac{x_0-t_1}\e}^{\frac{x_1-t_1}\e}\int_{\frac{x_1-t_2}\e}^{\frac{x_2-t_2}\e} \frac{|v^{(k)}(x)-
v^{(k)}(-y)|^2}{\big|y-x+\frac{t_2-t_1}\e \big|^{1+2s}}dy\,dx.
$$
Since $v^{(k)}(x)=v^{(k)}(-y)$ if $x\ge T$ and $y\le -T$ or  if $x\le -T$ and $y\ge T$ (note that if $k\ge 1$ this is true also if $x,y\ge T$ or $x,y\le -T$), it suffices to estimate the double integrals
$$
I^\e_1:=\int_{-T}^{T}\int_{\frac{x_1-t_2}\e}^{\frac{x_2-t_2}\e} \frac{|v^{(k)}(x)-
v^{(k)}(-y)|^2}{\big|y-x+\frac{t_2-t_1}\e \big|^{1+2s}}dy\,dx
$$
and 
$$
I^\e_2:=\int_{T}^{\frac{x_1-t_1}\e}\int_{-T}^{\frac{x_2-t_2}\e} \frac{|v^{(k)}(x)-
v^{(k)}(-y)|^2}{\big|y-x+\frac{t_2-t_1}\e \big|^{1+2s}}dy\,dx,
$$
and the analogous $I_3$ and $I_4$ with the role of $x$ and $y$ interchanged.

We first estimate
\begin{eqnarray}\label{Ie1}\nonumber
I^\e_1&\le& 2\int_{-T}^{T}\int_{\frac{x_1-t_2}\e}^{\frac{x_2-t_2}\e} \frac{|v^{(k)}(x)|^2+
|v^{(k)}(-y)|^2}{\big|y-x+\frac{t_2-t_1}\e \big|^{1+2s}}dy\,dx
\\ \nonumber
&= &2\int_{-T}^{T}|v^{(k)}(x)|^2\Big(\int_{\frac{x_1-t_2}\e}^{\frac{x_2-t_2}\e} +
\frac1{\big|y-x+\frac{t_2-t_1}\e \big|^{1+2s}}dy\Big)\,dx
\\
&&+2\int_{-T}^{T}\Big(\int_{\frac{x_1-t_2}\e}^{\frac{x_2-t_2}\e} \frac{
|v^{(k)}(-y)|^2}{\big|y-x+\frac{t_2-t_1}\e \big|^{1+2s}}dy\Big)\,dx.
\end{eqnarray}
For $|x|\le T$ we have
$$
\int_{\frac{x_1-t_2}\e}^{\frac{x_2-t_2}\e} \frac{
1}{\big|y-x+\frac{t_2-t_1}\e \big|^{1+2s}}dy
\le\frac{\frac{t_2-t_1}\e}{\big|\frac{x_1-t_1}\e-T \big|^{1+2s}} =O(\e^{2s})
$$
as $\e\to 0$, and hence also
\begin{eqnarray*}
&&\hskip-1.5cm\int_{\frac{x_1-t_2}\e}^{\frac{x_2-t_2}\e} \frac{
|v^{(k)}(-y)|^2}{\big|y-x+\frac{t_2-t_1}\e \big|^{1+2s}}dy\ 
\le\ \int_{\frac{x_1-t_2}\e}^{-T} \frac{
1}{\big|y-x+\frac{t_2-t_1}\e \big|^{1+2s}}dy
\\
&&+\int_{-T}^{T} \frac{
|v^{(k)}(-y)|^2}{\big|y-x+\frac{t_2-t_1}\e \big|^{1+2s}}dy
+
\int_{T}^{\frac{x_2-t_2}\e} \frac{1}{\big|y-x+\frac{t_2-t_1}\e \big|^{1+2s}}dy
\\
&\le&
\int_{\frac{x_1-t_2}\e}^{\frac{x_2-t_2}\e} \frac{
1}{\big|y-x+\frac{t_2-t_1}\e \big|^{1+2s}}dy
+\int_{-T}^{T} \frac{
|v^{(k)}(-y)|^2}{\big|y-x+\frac{t_2-t_1}\e \big|^{1+2s}}dy
\\
&\le& \frac{\frac{x_2-x_1}\e}{\big|\frac{x_1-t_1}\e-T \big|^{1+2s}}+\frac1{\big|\frac{t_2-t_1}\e-2T \big|^{1+2s}}\int_{-T}^{T} 
|v^{(k)}(y)|^2dy=o(1)
\end{eqnarray*}
as $\e\to 0$. Using these estimates in \eqref{Ie1} we obtain that $I^\e_1=o(1)$ as $\e\to 0$. Note that the second estimate slightly improves if $k\ge 1$ since $v^{(k)}(y)=0$ if $|y|\ge T$. An analogous argument shows that $I^\e_2=o(1)$ as $\e\to 0$,
so that $I^\e$ tends to $0$ as $\e\to 0$. Note that if $i-j$ is even, the argument above can be followed word for word since we have only used the estimate $|v^{(k)}(x)-v^{(k)}(-y)|^2\le 2(|v^{(k)}(x)|^2+|v^{(k)}(-y)|^2)$, and hence the claim is proved.

We finally note that if $s>\frac12$ then the estimate of $I^\e_{ij}$ defined in \eqref{Iij}  for $i\neq j$ can be achieved by a simpler direct computation, as we have
$$
I^\e_{ij}\le\e^{2s-1} \int_{x_{i-1}}^{x_i}\int_{x_{j-1}}^{x_j}\frac{4\|v^{(k)}\|_\infty^2}{|x-y|^{1+2s}}dy\,dx\le \e^{2s-1}\frac{8\|v^{(k)}\|_\infty^2}{s(2s-1){\delta^{2s-1}}},
$$
where $\delta=\min\{|x_k-x_\ell|: k\neq \ell\}$.
\end{proof} 

The following proposition  concludes the proof of Theorem \ref{main}.

\begin{proposition}
Let $m_{k+s}$ and $\widetilde m_{k+s}$ be defined in \eqref{emme} and \eqref{emmetilde}, repectively. Then $m_{k+s}=\widetilde m_{k+s}$.
\end{proposition}

\begin{proof} By definition $m_{k+s}\le\widetilde m_{k+s}$ since all test functions for the problems defining $\widetilde m_{k+s}$ are test functions for the problem defining 
$m_{k+s}$.

In order to prove the converse inequality, let $v$ be a test function for the problem defining $m_{k+s}$. We can apply the cut-off method in the proof of Proposition~\ref{lo-bo} with $\e=1$ and $u_\e=v$, and $3\tau=T$ arbitrary. The corresponding function $w_\e=w_T$ (which in this case coincides with $v_\e$) does not depend on $\e$ but indeed depends on $T$, and \eqref{dis-co-1} reads
\begin{eqnarray*}\nonumber\label{dis-co-3}
&&\hspace{-1cm}\int_{\mathbb R} W(w_T)\, dx+\int_{\mathbb R}\!\int_{\mathbb R} \frac{|w_T^{(k)}(x)-w_T^{(k)}(y)|^2}{|x-y|^{1+2s}}\, dx\, dy\\
&&\hspace{3cm}\leq 
F_1(v;(-T,T))+ c_\sigma+o(1)
\end{eqnarray*} 
as $T\to+\infty$, with $c_\sigma$ arbitrarily small. Since the right-hand side tends to $F_1(v,\mathbb R)+c_\sigma$ we obtain the desired estimate.
\end{proof}

\section{Correction of the discontinuities at integer exponents}
We now examine the behaviour of the coefficients $m_{k+s}$, showing their degeneracy close to integer points. We then prove that this discontinuity can be overcome and  the functionals $F_\e$, after suitable modifications, can be interpreted as interpolations of corresponding local functionals depending on integer derivatives. 

We first examine the behaviour of the coefficients $m_r=m_{k+s}$ defined in \eqref{emme} with $k=\lfloor r\rfloor$ and $s= r-\lfloor r\rfloor$. The special case of $k=0$ and $s\to 1^-$ has been examined in \cite{PaVi}. In this section we make the technical hypothesis that $W(z)\ge c_1|z|^2 -c_2$, so that functionals $F_\e$ are equicoercive in $L^2$.  

\begin{proposition}\label{continuityofm1} The function $r\mapsto m_{r}$ is a continuous function on $(\frac12,+\infty)\setminus \mathbb N$ and
\begin{equation}\label{limeks}
\lim_{r\to k} m_{r}=  +\infty;
\end{equation}
that is, the coefficients blow up at integer points.
\end{proposition}

\begin{proof} We set $G_s(u)=[u]^2_s$ and we note that the functionals $G_{s+r}$ $\Gamma$-converge to $G_s$ as $r\to 0$ with respect to the $L^2_{\rm loc}$-convergence for all $s\in(0,1)$. The liminf inequality follows from an application of Fatou's Lemma.
Indeed if $u_r\to u$ in $L^2_{\rm loc}(I)$ with $G_{s+r}(u_r)\le S<+\infty$, with fixed $\delta>0$, letting $\Delta_\delta=\{(x,y): |x-y|\le\delta\}$, for all bounded intervals $I\subset \mathbb R$, we have 
$$
\int_{(I\times I)\setminus \Delta_\delta}\frac{|u^{(k)}(x)- u^{(k)}(y)|^2}{|x-y|^{1+2s}}\,dxdy\le \liminf_{r\to0}\int_{(I\times I)\setminus \Delta_\delta}\frac{|u^{(k)}(x)- u^{(k)}(y)|^2}{|x-y|^{1+2(s+r)}}\,dxdy
$$
$$
\le \liminf_{r\to0}[u_r]^2_{s+r}.
$$
Letting $\delta\to 0$ and using the arbitrariness of $I$ we have that $u\in H^s(I)$ and the lower bound. Conversely, if $u$ is $C^\infty$ and constant for $|x|\ge T$ we have the pointwise convergence $G_{s+r}(u)\to G_s(u)$, which proves the upper bound. 
Finally, we note that, since the argument showing the equi-boundedness of the number of transitions can be applied uniformly as $r\to 0$, up to translations we have that $u_r\to \pm1$ uniformly at $\pm\infty$, so that $u\to\pm1$ at $\pm\infty$. By the fundamental theorem of $\Gamma$-convergence, we then obtain the convergence of the related minimum problems; that is, $m_{s+r}\to m_s$.

Suppose now that, up to subsequences, we have $m_s\le C<+\infty$ as $s\to 1-$.
Then, since by the results of Ponce \cite{ponce} the functionals $(1-s)G_s$ $\Gamma$-converge to $G_{k+1}(u):=\int_\mathbb R |u^{(k+1)}|^2dt$.
If we take a sequence $u_s$ almost realizing the minimum of $m_s$ then, up to translations and subsequences $u_s\to u$ in $L^2_{\rm loc}(\mathbb R)$ we would have $G_{k+1}(u)\le \liminf_{s\to 1^-} G_s(u_s)=0$, so that $u^{(k+1)}=0$ and $u_s$ converges uniformly to a non-trivial $k$-th order polynomial, for which the integral $\int W(u_s)dx$ blows up.

An analogous argument can be used for $s\to 0$ by the $\Gamma$-convergence of $s\,G_s$ $\Gamma$-converge to $2G_{k}$ by \cite{MS}.
\end{proof}

In order to correct the behaviour of $m_{k+s}$ at integer points, for $s\in(0,1)$ we define the functionals
\begin {equation}
F^{k,s}_\e(u)=\frac1\e\int_I W(u(x))dx+\frac{s(1-s)}{2^{1-s}}\e^{2(k+s)-1}[u]^2_{k+s}(I),
\end{equation}
which differs by a multiplicative factor, singular at $s=0$ and $s=1$, in the higher-order singular perturbation.
These functionals are still of the form of $F_\e$, since they can be rewritten as
$$
F^{k,s}_\e(u)=2^{s-1}s(1-s)\Big(\frac1\e\int_I W_{k,s}(u(x))dx+\e^{2(k+s)-1}[u]^2_{k+s}(I)\Big),
$$
with $W_{k,s}(z)= 2^{1-s}\frac{1}{s(1-s)}W(z)$ still satisfying (H1) and (H2), and their $\Gamma$-limit has the coefficient $m_k(s)$ defined in \eqref{emmeks}; that is,
\[
 m_{k}(s) =\inf\bigg\{\int_{\mathbb R} W(v)dx+\tfrac{s(1-s)}{2^{1-s}}[v]^2_{k+s}(\mathbb R) :v\in H^{k+s}_{\rm loc}(\mathbb R), v(\pm\infty)=\pm1\bigg\}.
\]

\begin{proposition}\label{continuityofm2} For $k\in\mathbb N$ and $s\in(0,1)$, let $m_{k}(s)$ be defined by 
\eqref{emmeks}, and let 
\begin{equation}\label{emmek}
m_{k}=\inf\Big\{\int_{\mathbb R} W(v)dx+\int_\mathbb R|v^{(k)}|^2 dx : v\in H^{k}_{\rm loc}(\mathbb R), \lim_{x\to\pm\infty}v(x)=\pm1\Big\}.
\end{equation}
Then $s\mapsto m_k(s)$ is a continuous function and
\begin{equation}\label{limeks}
\lim_{s\to 0^+} m_k(s)= m_k,\qquad \lim_{s\to 1^-} m_k(s)= m_{k+1}.
\end{equation}
In particular, the function $r\mapsto m_{\lfloor r\rfloor}(r-\lfloor r\rfloor) $ is continuous on $(\frac12,+\infty)$.
\end{proposition}

\begin{proof} The continuity in $(0,1)$ follows repeating the corresponding part of the proof of Proposition \ref{continuityofm1}. The continuity at $s=0$ and $s=1$ also follows from the proof of Proposition \ref{continuityofm1}, once we observe that the argument showing the equi-boundedness of transitions can be applied uniformly as $s\to 0^+$ and as $s\to 1^-$ on functions $u_s$ with bounded $F^{k,s}_1(u_s)$, and this guarantees that, up to translations and subsequences, $u_s\to u$ which is a test function for $m_k$ and $m_{k+1}$, respectively. The $\Gamma$-convergence results of \cite{MS} and \cite{ponce} allow us to conclude the proof.
\end{proof}

\noindent{\bf Acknowledgements.}
 This paper is based on work supported by the GNAMPA Project ``Asymptotic analysis of nonlocal variational problems'' funded by INdAM.  
The author is a member of GNAMPA of INdAM.

\bibliographystyle{abbrv}

\bibliography{references}

\begin{thebibliography}{10}

\bibitem{alberti}
G.~Alberti.
\newblock Variational models for phase transitions, an approach via
  {$\Gamma$}-convergence.
\newblock In {\em Calculus of {V}ariations and {P}artial {D}ifferential
  {E}quations ({P}isa, 1996)}, pages 95--114. Springer, Berlin, 2000.

\bibitem{AB}
G.~Alberti and G.~Bellettini.
\newblock A nonlocal anisotropic model for phase transitions. {I}. {T}he
  optimal profile problem.
\newblock {\em Math. Ann.}, 310(3):527--560, 1998.

\bibitem{ABS}
G.~Alberti, G.~Bouchitt\'e, and P.~Seppecher.
\newblock Un r\'esultat de perturbati\-ons singuli\`eres avec la norme
  {$H^{1/2}$}.
\newblock {\em C. R. Acad. Sci. Paris S\'er. I Math.}, 319(4):333--338, 1994.

\bibitem{ABS-2}
G.~Alberti, G.~Bouchitt\'e, and P.~Seppecher.
\newblock Phase transition with the line-tension effect.
\newblock {\em Arch. Rational Mech. Anal.}, 144(1):1--46, 1998.

\bibitem{AT}
L.~Ambrosio and V.~M. Tortorelli.
\newblock Approximation of functional depending on jumps by elliptic functional
  via {$\Gamma$}-convergence.
\newblock {\em Comm. Pure Appl. Math.}, 43(8):999--1036, 1990.

\bibitem{MR1051228}
S.~Baldo.
\newblock Minimal interface criterion for phase transitions in mixtures of
  {C}ahn--{H}illiard fluids.
\newblock {\em Ann. Inst. H. Poincar\'e{} C Anal. Non Lin\'eaire}, 7(2):67--90,
  1990.

\bibitem{MR1286918}
A.~C. Barroso and I.~Fonseca.
\newblock Anisotropic singular perturbations---the vec\-to\-rial case.
\newblock {\em Proc. Roy. Soc. Edinburgh Sect. A}, 124(3):527--571, 1994.

\bibitem{MR1036589}
G.~Bouchitt\'e.
\newblock Singular perturbations of variational problems arising from a
  two-phase transition model.
\newblock {\em Appl. Math. Optim.}, 21(3):289--314, 1990.

\bibitem{BBM}
J.~Bourgain, H.~Brezis, and P.~Mironescu.
\newblock Another look at {S}obolev spaces.
\newblock In {\em Optimal {C}ontrol and {P}artial {D}ifferential {E}quations},
  pages 439--455. IOS, Amsterdam, 2001.

\bibitem{bln}
A.~Braides.
\newblock {\em Approximation of Free-discontinuity Problems}.
\newblock Springer-Verlag, Berlin, 1998.

\bibitem{BCST}
A.~Braides, A.~Causin, M.~Solci, and L.~Truskinovsky.
\newblock Beyond the classical {C}auchy-{B}orn rule.
\newblock {\em Arch. Ration. Mech. Anal.}, 247(6):Paper No. 107, 113, 2023.

\bibitem{BDS}
G.~C. Brusca, D.~Donati, and M.~Solci.
\newblock Higher-order singular perturbation models for phase transitions,
  2024. https://arxiv.org/abs/2402.13626.

\bibitem{CH}
J.~W. Cahn and J.~E. Hilliard.
\newblock Free energy of a nonuniform system. {I}. {I}n\-ter\-fa\-cial free
  energy.
\newblock {\em J. Chem. Phys.}, 28(2):258--267, 1958.

\bibitem{CKP}
V.~Crismale, L.~De~Luca, A.~Kubin, A.~Ninno, and M.~Ponsiglione.
\newblock The variational approach to {$s$}-fractional heat flows and the limit
  cases {$s\to0^+$} and {$s\to1^-$}.
\newblock {\em J. Funct. Anal.}, 284(8):Paper No. 109851, 38, 2023.

\bibitem{MR3748585}
G.~Dal~Maso, I.~Fonseca, and G.~Leoni.
\newblock Asymptotic analysis of second order nonlocal {C}ahn--{H}illiard-type
  functionals.
\newblock {\em Trans. Amer. Math. Soc.}, 370(4):2785--2823, 2018.

\bibitem{DPV}
E.~{Di~Nezza}, G.~Palatucci, and E.~Valdinoci.
\newblock Hitchhiker's guide to the fractional {S}obolev spaces.
\newblock {\em Bull. Sci. Math.}, 136:521--573, 2012.

\bibitem{DipVal}
S.~Dipierro and E.~Valdinoci.
\newblock Some perspectives on (non)local phase transitions and minimal
  surfaces.
\newblock {\em Bull. Math. Sci.}, 13(1):Paper No. 2330001, 77, 2023.

\bibitem{FM}
I.~Fonseca and C.~Mantegazza.
\newblock Second order singular perturbation models for phase transitions.
\newblock {\em SIAM J. Math. Anal.}, 31(5):1121--1143, 2000.

\bibitem{FMu}
I.~Fonseca and S.~M\"uller.
\newblock Quasi-convex integrands and lower semicontinuity in {$L^1$}.
\newblock {\em SIAM J. Math. Anal.}, 23(5):1081--1098, 1992.

\bibitem{gurtin}
M.~E. Gurtin.
\newblock Some results and conjectures in the gradient theory of phase
  transitions.
\newblock In {\em Metastability and Incompletely Posed Problems}, pages
  135--146. Springer, 1987.

\bibitem{Landau}
L.~Landau.
\newblock On the theory of phase transitions.
\newblock {\em Zh. Eksp. Teor. Fiz.}, 7:19--32, 1937.

\bibitem{leofrac}
G.~Leoni.
\newblock {\em A {F}irst {C}ourse in {F}ractional {S}obolev {S}paces}, volume
  229 of {\em Graduate Studies in Mathematics}.
\newblock American Mathematical Society, Providence, RI, 2023.

\bibitem{MS}
V.~Maz'ya and T.~Shaposhnikova.
\newblock On the {B}ourgain, {B}rezis, and {M}ironescu theorem concerning
  limiting embeddings of fractional {S}obolev spaces.
\newblock {\em J. Funct. Anal.}, 195(2):230--238, 2002.

\bibitem{Modica}
L.~Modica.
\newblock The gradient theory of phase transitions and the minimal interface
  criterion.
\newblock {\em Arch. Ration. Mech. Anal.}, 98:123--142, 1987.

\bibitem{MM}
L.~Modica and S.~Mortola.
\newblock Un esempio di {$\Gamma \sp{-}$}-convergenza.
\newblock {\em Boll. Un. Mat. Ital. B (5)}, 14(1):285--299, 1977.

\bibitem{MR1097327}
N.~C. Owen and P.~Sternberg.
\newblock Nonconvex variational problems with an\-iso\-tro\-pic perturbations.
\newblock {\em Nonlinear Anal.}, 16(7-8):705--719, 1991.

\bibitem{PaVi}
G.~Palatucci and S.~Vincini.
\newblock Gamma-convergence for one-dimensional nonlocal phase transition
  energies.
\newblock {\em Matematiche (Catania)}, 75(1):195--220, 2020.

\bibitem{ponce}
A.~C. Ponce.
\newblock A new approach to {S}obolev spaces and connections to
  {$\Gamma$}-con\-ver\-gence.
\newblock {\em Calc. Var. Partial Differential Equations}, 19(3):229--255,
  2004.

\bibitem{SV}
O.~Savin and E.~Valdinoci.
\newblock {$\Gamma$}-convergence for nonlocal phase transitions.
\newblock {\em Ann. Inst. H. Poincar\'e{} C Anal. Non Lin\'eaire},
  29(4):479--500, 2012.

\bibitem{solci}
M.~Solci.
\newblock Local interpolation techniques for higher-order singular
  per\-tur\-ba\-tions of non-convex functionals: free-discontinuity problems,
  2024. https://arxiv.org/abs/2402.10656.

\bibitem{vdW}
J.~van~der Waals.
\newblock The thermodynamic theory of capillarity under the hy\-po\-thesis of a
  continuous variation of density (in {D}utch).
\newblock {\em Verhandel. Ko\-nink. Akad. Weten. Amsterdam (Section 1)}, 1(8),
  1893.

\end{thebibliography}

\end{document}